\documentclass[12pt]{article}
\usepackage{graphicx} 
\usepackage{float}    
\usepackage{verbatim} 
\usepackage{amsmath}  
\usepackage{amssymb}  
\usepackage{subfig}   
\usepackage[colorlinks=true,citecolor=blue,linkcolor=blue,urlcolor=blue]{hyperref}
\usepackage{fullpage}
\usepackage{amsthm,mathtools}
\usepackage{enumerate}
\usepackage{paralist}
\usepackage{xspace}
\usepackage{caption}
\usepackage{bbm}
\usepackage{url}
\usepackage{mathrsfs}
\usepackage{algorithm}
\usepackage{algorithmic}

\newcommand{\vertiii}[1]{{\left\vert\kern-0.25ex\left\vert\kern-0.25ex\left\vert #1 
		\right\vert\kern-0.25ex\right\vert\kern-0.25ex\right\vert}}

\makeatletter

\newtheorem{theorem}{Theorem}
\newtheorem{definition}{Definition}
\newtheorem{lemma}{Lemma}

\newtheorem{remark}{Remark}

\newtheorem{proposition}{Proposition}

\newtheorem{assumption}{Assumption}

\usepackage[margin=1in]{geometry}
\usepackage{times}

\usepackage{thmtools}
\usepackage{thm-restate}

\usepackage{listings}

\@ifundefined{showcaptionsetup}{}{%
	\PassOptionsToPackage{caption=false}{subfig}}
\usepackage{subfig}
\makeatother

\usepackage{listings}

\global\long\def\bDelta{\mathbf{\Delta}}
\global\long\def\bPhi{\mathbf{\Phi}}
\global\long\def\H2{\mathcal{H}_2}
\global\long\def\Hinf{\mathcal{H}_{\infty}}
\global\long\def\bK{\mathbf{K}}
\global\long\def\bV{\mathbf{V}}

\global\long\def\cC{\mathcal{C}}

\global\long\def\E1{\mathcal{E}_1}

\global\long\def\trueK{\mathbf{K}_\star}

\begin{document}

	\title{\bf Efficient Learning of Distributed Linear-Quadratic Controllers}
	\author{Salar Fattahi$^1$,  Nikolai Matni$^2$, Somayeh Sojoudi$^1$
		\date{%
			$^1$University of California, Berkeley\\%
			$^2$University of Pennsylvania%
		}
	}

	\maketitle
	
	\begin{abstract}
		In this work, we propose a robust approach to design distributed controllers for unknown-but-sparse linear and time-invariant systems. By leveraging modern techniques in distributed controller synthesis and structured linear inverse problems as applied to system identification, we show that near-optimal distributed controllers can be learned with sub-linear sample complexity and computed with near-linear time complexity, both measured with respect to the dimension of the system. In particular, we provide sharp end-to-end guarantees on the stability and the performance of the designed distributed controller and prove that for sparse systems, the number of samples needed to guarantee robust and near optimal performance of the designed controller can be significantly smaller than the dimension of the system. Finally, we show that the proposed optimization problem can be solved to global optimality with near-linear time complexity by iteratively solving a series of small quadratic programs.
	\end{abstract}

\section{Introduction}\label{sec:problem}
Encouraged by the success of machine learning (ML) techniques applied to complex decision making problems \cite{jordan2015machine} such as image classification \cite{krizhevsky2012imagenet}, video and board games  \cite{atari-nature,alphago,silver2018general}, and robotics \cite{duan2016benchmarking,openAI,Levine16}, the use of ML for the control of autonomous systems interacting with physical environments has been an active area of research in recent years.  While there is an increasing body of work studying the theoretical and practical aspects of deploying learning-enabled control policies in individual systems (e.g., self-driving cars, agile robots) \cite{Levine16,duan2016benchmarking,bojarski2016end,stulp2012reinforcement,recht2019tour}, there has been little work studying the use of these techniques on \emph{distributed systems}, that is to say systems composed of interconnected and often spatially-distributed subsystems. Examples of such distributed systems include intelligent transportation systems and cities, smart grids, and distributed sensor networks.  Even when the individual components are well modeled, controlled, and understood, integrating them into a large-scale, interconnected, and heterogeneous system can make modeling and control of the full system challenging, strongly motivating the use of machine-learning-based techniques.

Extending the application of data-driven techniques to large-scale and safety-critical systems requires overcoming several challenges.  First, we must ensure that the new data-driven methods lead to autonomous systems that are safe, reliable, and robust, as many of our target application areas correspond to safety-critical infrastructure.  Failure of such systems could be catastrophic in terms of both social, economic, and possible human losses.  Second, any proposed learning and control algorithm must scale gracefully to large-scale and potentially spatially distributed systems.  To address these challenges, we extend the approach taken in \cite{dean2017sample} for designing centralized control policies to the \emph{distributed} optimal control of an unknown \emph{distributed} dynamical system.  We develop both deterministic and probabilistic guarantees for a novel robust distributed control synthesis approach. Our proposed method is scalable to large systems, and it allows us to provide the first end-to-end sample complexity guarantees for the distributed optimal control of an unknown system.

In particular, we consider the discrete-time stochastic linear time-invariant system
\begin{align}\label{dyn}
x(t+1) = A_\star x(t)+B_\star u(t)+w(t)
\end{align}
with the state $x(t)\in\mathbb{R}^n$, {state matrix} $A_\star\in\mathbb{R}^{n\times n}$, controllable input $u(t)\in\mathbb{R}^m$, {input matrix} $B_\star\in\mathbb{R}^{n\times m}$, and exogenous random noise $w(t)\in\mathbb{R}^n$ (also referred to as disturbance noise). The goal is to design a control policy $u(t) = f(\{x(\tau)\}_{\tau=0}^t,\{u(\tau)\}_{t=0}^\tau)$ that minimizes the following expected cost function:
\begin{align}\label{cost}
\lim\limits_{T\to\infty}\frac{1}{T}\sum_{t=1}^{T}\mathbb{E}\left\{x(t)^\top Qx(t)+u(t)^\top Ru(t)\right\}
\end{align}
subject to dynamics~\eqref{dyn}, where $Q$ and $R$ are positive-definite matrices. When the system matrices are known and there is no communication constraint on the control policy, this problem reduces to the well-known centralized linear-quadratic regulator (LQR) design for which the static linear policy $u(t) = Kx(t)$ is known to be optimal. The optimality of this control policy is contingent upon the full knowledge of the system matrices, as well as the absence of communication constraints on the structure of the controller. However, these conditions are not satisfied in general, as the system may be subject to unknown dynamics and spatiotemporal constraints as discussed below:

\vspace{1mm}
\noindent\textbf{Unknown dynamics:} In many systems, the exact parameters of the dynamics are not known \textit{a priori}. In particular, rather than having direct access to the system matrices $(A_\star,B_\star)$, we usually only have access to some estimates $(\hat{A},\hat{B})$ obtained from first principles, domain knowledge, or a system identification technique.  Further, in the distributed setting, the \emph{sparsity structure} of these matrices may be unknown as well, due to dynamic interconnections between component sub-systems.  As we describe in the sequel, identifying a structured model is the key to scaling robust and optimal control methods to large systems.

\noindent\textbf{Spatiotemporal constraints:}
Large-scale distributed systems, such as power grids and distributed computing networks, are composed of smaller sub-systems that are locally interconnected according to a physical interaction topology.  Exploiting the underlying sparsity of these systems, as induced by the local interactions between subsystems, is crucial in extending robust and optimal control methods to the distributed setting \cite{wang2016system,2006_Rotkowitz_QI_TAC} by allowing \emph{local} sub-controllers to communicate and coordinate with each other. Furthermore, from a practical perspective,  controllers that can be implemented using \emph{finite impulse response (FIR)} components lead to simple and intuitive implementations \cite{wang2014localized,wang2017separable}.

\vspace{1mm}


\subsection{Contributions}
In this work, we overcome the aforementioned difficulties by leveraging recent advances in control theory and machine learning. Namely, we develop a novel distributed robust control synthesis method using the System Level Synthesis (SLS) framework \cite{wang2016system}, and combine it with model error bounds obtained via the non-asymptotic analysis of regularized estimators as applied to sparse system identification \cite{Salar18,fattahi2019learning}, leading to a method that is efficient \textit{both in sample and computational complexities.}

Given the estimates $(\hat{A},\hat{B})$ of the \textit{true} system matrices $(A_\star,B_\star)$, we are interested in designing a \textit{distributed} controller that can guarantee the stability of the true system with a small optimality gap in its cost function. In particular, given the estimates $\hat{A}$, $\hat{B}$ with an estimation error $\epsilon:=\max\{\|\hat{A}-A_\star\|_2, \|\hat{B}-B_\star\|_2\}$, we propose a method to design a dynamic and linear state-feedback controller $\bK$ that 1) admits a distributed implementation, respecting the spatiotemporal constraints imposed by the underlying communication topology, and 2) is robust against the model uncertainties; in particular, it stabilizes the closed-loop gain $A_\star+B_\star\bK$ and admits a relative sub-optimality bound $J(A_\star, B_\star, \bK)-J_\star\leq \alpha(\epsilon, L)J_\star$ for some positive sub-optimality factor $\alpha(\epsilon, L)$. Here, $J(A_\star, B_\star, \bK)$ is the value of the cost function~\eqref{cost} achieved by the controller  $\mathbf{u} = \bK\mathbf{u}$ acting on the true system, and $J_\star$ is the cost of the \textit{oracle} distributed controller to be formally defined later. Furthermore, $L$ is the enforced temporal length of the obtained system responses with the designed controller. We show that the sub-optimality factor $\alpha(\epsilon, L)$ can be decomposed into two terms:
\begin{align}
\alpha(\epsilon, L) = \alpha_e(\epsilon)+\alpha_t(L)
\end{align}
\begin{sloppypar}
	\noindent	where $\alpha_e(\epsilon)$ bounds the performance degradation caused by \emph{model uncertainty}, and $\alpha_t(L)$ bounds the effect of \textit{temporal truncation}, which quantifies the deviation of the designed controller from its oracle counterpart, when the system responses are restricted to the FIR filters with length $L$. We prove that the uncertainty and truncation errors decay linearly in $\epsilon$ and exponentially in $L$, respectively. Furthermore, by carefully examining the sparsity structure of the estimated system matrices and the controller, we show that under some conditions, these errors \textit{do not} scale with the system dimensions, and instead, they are only dependent on the sparsity structures of the system dynamics and the controller, as well as other spectral characteristics of the system. By combining the derived bounds with the recent high-dimensional system identification techniques \cite{Salar18,fattahi2019learning}, we provide an end-to-end sub-optimality bound on the performance of the designed distributed controller in terms of the number of sample trajectories that are used for estimating the dynamics, as well as the required temporal length of the system responses. Finally, we provide an efficient algorithm with near-linear time complexity to solve the proposed optimization problem. The performance of the presented method is extensively evaluated in different case studies.
\end{sloppypar}

\vspace{2mm}
\noindent{\bf Notation:} Upper- and lower-case letters are used to denote matrices and vectors, respectively. Boldface upper- and lower-case letters refer to transfer matrices and vector-valued signals, respectively. For a matrix $M$, the symbols $\|M\|_2$, $\|M\|_1$, and $\|M\|_{\infty}$ refer to its induced spectral norm, induced norm-1, and maximum absolute value of its elements, respectively. The symbols $\mathcal{H}_2$ and $\mathcal{H}_\infty$ are endowed with the standard definitions of the Hardy spaces, and $\mathcal{RH}_2$ and $\mathcal{RH}_\infty$ correspond to the restriction of these spaces to the set of real, rational, and proper transfer functions. For a transfer matrix $\mathbf{M}\in\mathcal{RH}_\infty$, one can write $\mathbf{M} = \sum_{\tau=0}^{\infty} M(\tau)z^{-\tau}$, where $M(\tau)$ is the $\tau^{\text{th}}$ spectral component of $\bf M$. The notation $x\sim \mathcal{N}(\mu, \Sigma)$ implies that $x$ is a random vector drawn from a Gaussian distribution with mean vector $\mu$ and covariance matrix $\Sigma$. Given a matrix $M$, the symbol $\mathrm{supp}(M)$ refers to a binary matrix that shares the same sparsity pattern as $M$. Finally, given a matrix $M_0$, the set $\mathcal{S}(M_0)$ is defined as $\{M\ |\  \mathrm{supp}(M) = \mathrm{supp}(M_0)\}$.
\subsection{Related Work}
\paragraph{Distributed Control}
Many dynamical systems, such as the power grid, intelligent transportation systems, and distributed computing networks, are large-scale, physically distributed, and interconnected.  In such settings, control systems are composed of several sub-controllers, each equipped with their own sensors and actuators -- these sub-controllers then exchange local sensor measurements and control actions via a communication network.  This information exchange between sub-controllers is constrained by the underlying properties of the communication network, ultimately manifesting as information asymmetry among sub-controllers.  This information asymmetry is what makes distributed optimal controller synthesis challenging \cite{1972_Ho_info,2012_Mahajan_Info_survey,2006_Rotkowitz_QI_TAC,2002_Bamieh_spatially_invariant,2005_Bamieh_spatially_invariant,2013_Nayyar_common_info}---indeed, early negative results gave reason to suspect that the resulting distributed optimal control problems were intractable \cite{1968_Witsenhausen_counterexample,1984_Tsitsiklis_NP_hard}.

However, in the early 2000s, a body of work \cite{2002_Bamieh_spatially_invariant, 2004_QI, 2004_Dullerud, 2005_Bamieh_spatially_invariant, 2006_Rotkowitz_QI_TAC, 2012_Mahajan_Info_survey, 2013_Nayyar_common_info} culminating with the introduction of quadratic invariance (QI) in the seminal paper \cite{2006_Rotkowitz_QI_TAC}, showed that for a large class of practically relevant systems, the resulting distributed optimal control problem is convex.  The identification of QI as a useful condition for determining the tractability of a distributed optimal control problem led to an explosion of synthesis results in this area \cite{2012_Lessard_two_player,2010_Shah_H2_poset,2013_Lamperski_H2,2013_Lessard_structure,2013_Scherer_Hinf,2014_Lessard_Hinf,2014_Matni_Hinf,2014_Tanaka_Triangular,2014_Lamperski_state, fattahi2019transformation}.  These results showed that the robust and optimal control methods that were proven so powerful for centralized systems could be used in distributed settings.  However, they also made clear that the synthesis and implementation of QI distributed optimal controllers did not scale gracefully with the size of the underlying system---indeed, the complexity of computing a QI distributed optimal controller is at least as expensive to compute as its centralized counterpart, and can be more difficult to implement.
This lack of scalability motivated the development of the SLS framework \cite{wang2016system}, which allowed for the convex synthesis of \emph{localized} distributed optimal controllers \cite{wang2014localized,wang2017separable} that enjoyed \emph{order constant} synthesis and implementation complexity. In this paper, we build upon the SLS framework to synthesize an efficient learning-based distributed controller.


\paragraph{System Identification} Estimating system models from input/output experiments has a well-developed theory dating back to the 1960s, particularly in the case of linear and time-invariant systems. Standard reference textbooks on the topic include \cite{aastrom1971system,ljung1999system,chen2012identification,goodwin1977dynamic}, all focusing on establishing \emph{asymptotic} consistency of the proposed estimators.   

On the other hand, contemporary results in statistical learning as applied to system identification seek to characterize \emph{finite time and finite data} rates, leaning heavily on tools from stochastic optimization and concentration of measure.  Such finite-time guarantees provide estimates of both system parameters and their uncertainty, which allows for a natural bridge to robust/optimal control.  
In \cite{dean2017sample}, it was shown that under full state observation, if the system is driven by Gaussian noise, the ordinary least squares estimate of the system matrices constructed from independent data points achieves order optimal rates that are linear in the system dimension. This result was later generalized to the single trajectory setting for (i) marginally stable systems in \cite{simchowitz2018learning}, (ii) unstable systems  in \cite{sarkar2018fast},
and (iii) partially observed stable systems in \cite{oymak2018non,sarkar2019finite,tsiamis2019finite,simchowitz2019learning}.  

In this paper, we leverage recent analogous results for the identification of sparse state-space parameters \cite{fattahi2018data,fattahi2019learning}, where rates are shown to be logarithmic in the ambient dimension, and polynomial in the number of nonzero elements to be estimated.  We note that \cite{fattahi2019learning} builds on \cite{pereira2010learning}, wherein non-asymptotic guarantees on the identification of sparse autonomous dynamical systems are established.

\paragraph{Machine Learning for Continuous Control}  We focus on classical and contemporary results most related to the approach taken in this paper.  The use of learning and adaptation in controller design goes back to Kalman: in particular, self-tuning adaptive control, as pioneered in \cite{kalman1958design,aastrom1973self}, proved to be successful, and was followed by a long sequence of contributions to adaptive control theory, deriving conditions for convergence, stability, robustness and performance under various assumptions.  Contemporary approaches can be viewed as non-asymptotic refinements of these classical problems.  The modern study of adaptive control, as applied to the LQR problem, was initiated in~\cite{abbasi2011regret}, which provided regret bounds for the optimal LQR control of an unknown system.  The work~\cite{abbasi2011regret} uses an Optimism in the Face of Uncertainty (OFU) based approach, where it maintains confidence ellipsoids of system parameters and selects those parameters that lead to the best closed-loop performance. This work was followed up by several refinements and extensions to different settings \cite{russo17,abeille18,ouyang17,abbasi2019model,dean2018regret,mania2019certainty,rantzer2018concentration}, and can all be viewed as model-based reinforcement learning algorithms. Another approach was taken in~\cite{aswani2013provably}, where the authors proposed a learning-based model predictive control (MPC) approach to guarantee the robustness and high performance of an unknown system.

Closest to our work are the results in \cite{dean2017sample}, where the LQR optimal control of an unknown system is studied in the centralized setting.  In \cite{dean2017sample}, the authors propose a two-step procedure.  First, they identify a coarse model of the matrices $(A_\star,B_\star)$ describing system behavior, as well as high-probability bounds on the corresponding model estimate uncertainty.  They then use these model and uncertainty estimates to synthesize a robustly stabilizing controller, and analyze the end-to-end sample complexity of the resulting controller performance. We generalize this approach to distributed settings, by efficiently exploiting the structure of the system both during the identification and control synthesis phase. This in turn allows us to reduce both the sample and computational complexities of learning distributed controllers, as will be described in the sequel.  

\section{Preliminaries on System Level Synthesis}\label{sec:prelim}
\newcommand{\trueA}{A_\star}
\newcommand{\trueB}{B_\star}
\newcommand{\tf}[1]{\mathbf{#1}}
\newcommand{\hinf}{\mathcal{H}_\infty}
\newcommand{\htwo}{\mathcal{H}_2}
\newcommand{\norm}[1]{\left\|#1\right\|}

Given the true system matrices, the optimal centralized LQR controller can be computed by solving its corresponding Ricatti equation~\cite{bertsekas1995dynamic}.  However, as described above, in general the resulting problem becomes highly difficult when solving for a structured controller since it amounts to an NP-hard problem~\cite{tsitsiklis1985complexity}. To circumvent this inherent difficulty,~\cite{wang2016system} introduces the SLS framework, and shows how it can be used to synthesize distributed controllers by optimizing over their induced closed-loop \emph{system responses}.

We motivate this approach via a simple example. Given a static state-feedback control policy $K$, the closed-{loop map from} the disturbance noise $\{w(0), w(1), \dots\}$ to the state $x(t)$ and the control input $u(t)$ at time $t$ is given by
\begin{equation}
\begin{array}{rcl}
x(t) &=& \sum_{\tau=0}^{t} (\trueA + \trueB K)^{\tau}w({t-\tau-1}) \:, \\
u(t) &=& \sum_{\tau=0}^t K(\trueA + \trueB K)^{\tau}w({t-\tau-1}) \:.
\end{array}
\label{eq:impulse-response}
\end{equation}
where, with a slight abuse of notation, the initial state $x(0)$ is denoted by $w(-1)$. Letting $\Phi_x(t) := (\trueA + \trueB K)^{t-1}$ and $\Phi_u(t) := K(\trueA + \trueB K)^{t-1}$, we can rewrite~\eqref{eq:impulse-response} as
\begin{equation}
\begin{bmatrix} x(t) \\ u(t) \end{bmatrix} =
\sum_{\tau=0}^t \begin{bmatrix}\Phi_x(\tau) \\ \Phi_u(\tau) \end{bmatrix}w({t-\tau-1}) \:,
\label{eq:phis}
\end{equation}
where {$\{\Phi_x(t),\Phi_u(t)\}$ are called the \emph{system responses} induced by the controller $K$. The closed-loop system response elements can be defined for a \textit{dynamic} controller in a similar vein. In particular, consider the control policy $\mathbf{u} = \bK\mathbf{x}$ for some dynamic controller $\bK$. Then, the closed-loop transfer matrices from the disturbance noise $\tf w$ to the state $\tf x$ and control action $\tf u$  satisfy

	\begin{equation}
	\begin{bmatrix} \tf x \\ \tf u \end{bmatrix} = \begin{bmatrix} (zI - A-B\tf K)^{-1} \\ \tf K (zI-A-B \tf K)^{-1} \end{bmatrix} \tf w.
	\label{eq:response}
	\end{equation}
	The following theorem parameterizes the set of stable closed-loop transfer matrices, as described in \eqref{eq:response}, that are achievable by any stabilizing controller $\tf K$.
	\begin{theorem}[State-Feedback Parameterization~\cite{wang2016system}]
		The followings are true:
		\begin{itemize}
			\item[-] The affine subspace defined by
			\begin{equation}
			\begin{bmatrix} zI - A & - B \end{bmatrix} \begin{bmatrix} \tf \Phi_x \\ \tf \Phi_u \end{bmatrix} = I, \ \tf \Phi_x, \tf \Phi_u \in \frac{1}{z}\mathcal{RH}_\infty
			\label{eq:achievable}
			\end{equation}
			parameterizes all system responses \eqref{eq:response} from $\tf w$ to $(\tf x, \tf u)$ that are achievable by an internally stabilizing state-feedback controller $\tf K$.
			\item[-] For any transfer matrices $\{\tf \Phi_x, \tf \Phi_u\}$ satisfying \eqref{eq:achievable}, the controller $\tf K = \tf \Phi_u \tf \Phi_x^{-1}$, as implemented in Figure \ref{fig:implementation}, is internally stabilizing and achieves the desired system response \eqref{eq:response}.
		\end{itemize}
		\label{thm:param}
	\end{theorem}
	
	\begin{figure}
		\centering
		\includegraphics[width=.5\textwidth]{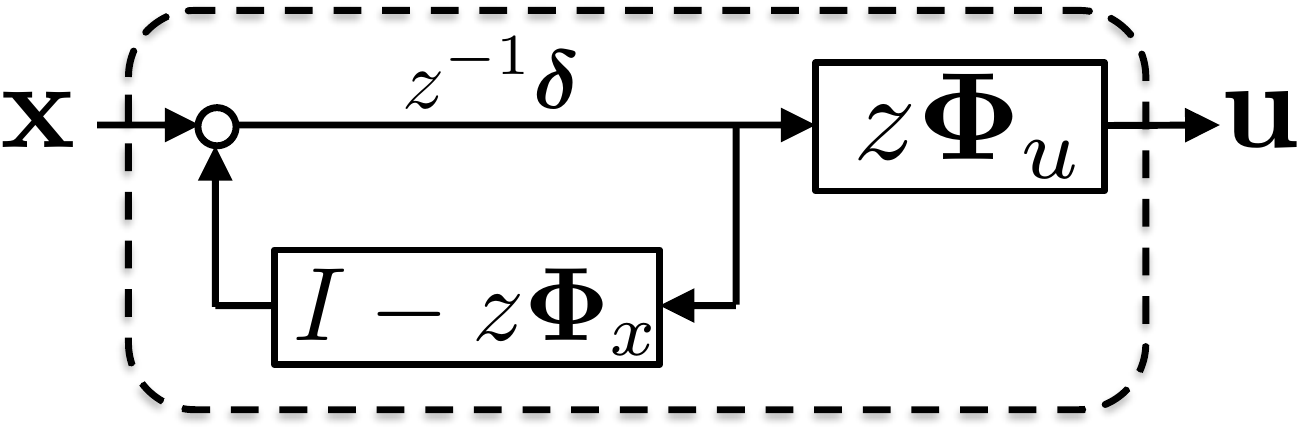}
		\caption{Internally stabilizing realization of the SLS controller specified in Theorem \ref{thm:param}.  Notice that sparsity structure imposed on the system responses $\{\tf \Phi_x, \tf \Phi_u\}$ translates directly to the \emph{internal sparsity structure} of the corresponding controller realization.}
		\label{fig:implementation}
	\end{figure}
	
	We now make two comments on the consequences of Theorem \ref{thm:param}.  
	First, note that $\{\tf \Phi_x, \tf \Phi_u\}=\{ (zI - A-B\tf K)^{-1} , \tf K (zI-A-B \tf K)^{-1} \}$ (as described in~\eqref{eq:response}) are elements of the affine subspace defined by~\eqref{eq:achievable} whenever $\tf K$ is a causal stabilizing controller. It is clear from~\eqref{eq:achievable} that any pair of transfer functions that satisfy~\eqref{eq:achievable} also obey
	\begin{equation}
	\Phi_x(t+1) = \trueA \Phi_x(t) + \trueB \Phi_u(t) \:, \:\: \Phi_x(1) = I \:, \:\: \forall t \geq 1 \:,
	\label{eq:time-achievability}
	\end{equation}
	and hence, satisfy the state-space equation. Furthermore, the above theorem implies that there exists a dynamic controller $\bK$ that achieves these system responses. The SLS framework therefore allows for any optimal control problem over linear systems to be cast as an optimization problem over elements $\{\Phi_x(t),\Phi_u(t)\}$, constrained to satisfy the affine equations~\eqref{eq:time-achievability}. Comparing equations \eqref{eq:impulse-response} and \eqref{eq:phis}, we see that the former is non-convex in the controller $\bK$, whereas the latter is convex in the elements $\{\Phi_x(t),\Phi_u(t)\}$, enabling solutions to previously difficult optimal control problems. 
	
	Second, notice that the realization of the controller $\tf K = \tf \Phi_u \tf \Phi_x^{-1}$ in Figure~\ref{fig:implementation} implies that any sparsity structure imposed on the the system responses translates directly to the internal structure of the corresponding controller.  Therefore, we can synthesize controllers that admit distributed realizations by imposing appropriate structural constraints on the system responses.  For example, if we wish to limit communications between sub-controllers that are first neighbors according to the topology defined by $\emph{A}$, it suffices to impose additional \emph{linear constraints} that the supports of the system responses $\tf \Phi_x$ and  $\tf \Phi_u$ be contained in the support of the matrix $A$. 
	This concept of \emph{locality} in system behavior and corresponding controller implementation is formalized and generalized in \cite{wang2014localized,wang2017separable}, and is the key in scaling robust and optimal control methods to large-scale distributed systems.  
	
	It follows from Theorem \ref{thm:param} and the standard equivalence between infinite horizon LQR and $\htwo$ optimal control that, for a disturbance process $w_t \overset{iid}{\sim{}} \mathcal{N}(0,\sigma_w^2 I)$, the standard LQR problem can be equivalently written as
	\begin{equation}
	\min_{\tf\Phi_x, \tf\Phi_u} \sigma_w^2 \left\|\begin{bmatrix} Q^\frac{1}{2} & 0 \\ 0 & R^\frac{1}{2}\end{bmatrix}\begin{bmatrix} \tf \Phi_x \\ \tf \Phi_u \end{bmatrix}\right\|_{\htwo}^2 \text{ s.t. equation \eqref{eq:achievable}}.
	\label{eq:lqr2}
	\end{equation}
	We drop the $\sigma_w^2$ in the objective function as it affects neither the optimal controller nor the sub-optimality guarantees. 
	
	Finally, we will make extensive use of a robust variant of Theorem \ref{thm:param}.
	\begin{sloppypar}
		\begin{theorem}[Robust Stability~\cite{virtual}]
			Suppose that the transfer matrices $\{\tf\Phi_x, \tf \Phi_u\} \in \frac{1}{z}\mathcal{RH}_\infty$ satisfy
			\begin{equation}
			\begin{bmatrix} zI - A & - B \end{bmatrix} \begin{bmatrix} \tf \Phi_x \\ \tf \Phi_u \end{bmatrix} = I + \tf \Delta.
			\end{equation}
			Then, the controller $\tf K = \tf \Phi_u \tf \Phi_x^{-1}$ stabilizes the system described by $(A,B)$ if and only if $(I+\tf \Delta)^{-1} \in \mathcal{RH}_\infty$.  Furthermore, the resulting system response is given by
			\begin{equation}
			\begin{bmatrix} \tf x \\ \tf u \end{bmatrix} = \begin{bmatrix} \tf \Phi_x \\ \tf \Phi_u \end{bmatrix}(I+\tf \Delta)^{-1} \tf w.
			\end{equation}
			\label{thm:robust}
		\end{theorem}
	\end{sloppypar}
	\section{Main Results}
	
	Following the SLS framework, the following optimization serves as an alternative formulation of the optimal distributed control problem:
	\begin{align}\label{eq22}
	\min_{\bPhi_x, \bPhi_x}&\left\|\begin{bmatrix}
	Q^{1/2} & 0\\
	0 & R^{1/2}
	\end{bmatrix}\begin{bmatrix}
	\bPhi_x\\
	\bPhi_u
	\end{bmatrix}\right\|_{\H2}\\
	\text{s.t.}& \begin{bmatrix}
	zI-{A} & -{B}
	\end{bmatrix}\begin{bmatrix}
	\bPhi_x\\
	\bPhi_u
	\end{bmatrix} = I,\\
	& \bPhi_x\in\frac{1}{z}\mathcal{R}\mathcal{H}_{\infty}\cap\cC_x,\\
	& \bPhi_u\in\frac{1}{z}\mathcal{R}\mathcal{H}_{\infty}\cap\cC_u,
	\end{align}
	where $\cC_x := \{\cC_x(\tau)\}_{\tau = 1}^\infty$ and $\cC_u := \{\cC_u(\tau)\}_{\tau = 1}^\infty$ capture the structural constraints on $\bPhi_x$ and $\bPhi_u$, respectively. In particular, we have $\Phi_x(\tau)\in \cC_x(\tau)$ and $\Phi_u(\tau)\in \cC_u(\tau)$ for every $\tau \in \{1,\dots,\infty\}$. 
	The optimization~\eqref{eq22} is referred to as the \textit{oracle optimization} and its corresponding optimal objective value is called the \textit{oracle cost}. 
	According to Theorem~\ref{thm:param}, the corresponding oracle controller $\bK^\star = \bPhi^\star_u{\bPhi^\star_x}^{-1}$ uniformly asymptotically stabilizes the true system. This together with the fact that for LTI systems, uniform asymptotic stability is equivalent to exponential stability, implies that the system responses are exponentially stable~\cite{wang2016system}. Therefore, upon writing $\bPhi^\star_x = \sum_{t=1}^{\infty}\Phi^\star_x(t)z^{-t}$ and $\bPhi^\star_u = \sum_{t=1}^{\infty}\Phi^\star_u(t)z^{-t}$, there exist constants $C_\star\geq 1$ and $0<\rho_\star<1$ such that
	\begin{equation}\label{stab_opt}
	\max\left\{\|\Phi^\star_x(t)\|_{\infty}, \|\Phi^\star_u(t)\|_{\infty}\right\}\leq C_\star\rho_\star^t
	\end{equation}
	for every integer $t$. 
	
	
	Note that in general the oracle optimization problem~\eqref{eq22} cannot be solved directly as it is infinite dimensional, and requires perfect knowledge of the system matrices $(\trueA,\trueB)$.  To circumvent this issue, we introduce a surrogate  to the oracle optimization that can be solved to robustly design a stabilizing distributed controller based on learned estimates $(\hat{A}, \hat{B})$, taking into account the resulted estimation error. Throughout the paper, $\epsilon$ is used to refer to the spectral norm of the estimation error. In particular, upon defining $\Delta_A = \hat{A}-\trueA$ and $\Delta_B = \hat{B}-\trueB$, we have $\epsilon := \max\{\|\Delta_A\|_{2},\|\Delta_B\|_{2}\}$.   We now recall a robust stability result from \cite{dean2017sample}:

	\begin{lemma}[\cite{dean2017sample}]\label{lem_cost}
		Suppose that the controller $\hat{\bK}$ stabilizes the system defined by the matrices $(\hat{A}, \hat{B})$ and that $(\hat{\bPhi}_x, \hat{\bPhi}_u)$ is its corresponding system response on $(\hat{A}, \hat{B})$. Then, controller $\hat{\bK}$ stabilizes the system defined by the matrices $(\trueA, \trueB)$ if $\|\hat{\bDelta}\|_{\Hinf}<1$, where
		\begin{align}
		\hat{\bDelta} = \begin{bmatrix}
		\Delta_A & \Delta_B
		\end{bmatrix}\begin{bmatrix}
		\hat{\bPhi}_x\\
		\hat{\bPhi}_u
		\end{bmatrix}.
		\end{align}
		Moreover, under this stability condition, one can write
		\begin{align}
		J(\trueA,\trueB,\hat{\bK}) = \left\|\begin{bmatrix}
		Q^{1/2} & 0\\
		0 & R^{1/2}
		\end{bmatrix}\begin{bmatrix}
		\hat{\bPhi}_x\\
		\hat{\bPhi}_u
		\end{bmatrix}\left(I+\hat{\bDelta}\right)^{-1}\right\|_{\H2}
		\end{align}
	\end{lemma}
	
	Following~\cite{dean2017sample}, we design a near-optimal distributed controller by solving the following robust counterpart of the oracle optimization problem~\eqref{eq22} based on the estimated values of $(\hat{A},\hat{B})$ with a given estimation error $\epsilon$:
	\begin{align}\label{eq1}
	\min_{\bPhi_x, \bPhi_x}\underset{\begin{subarray}
		.\|\Delta_A\|_{2}\leq\epsilon,\\ \|\Delta_B\|_{2}\leq \epsilon
		\end{subarray}}{\max}&\left\|\begin{bmatrix}
	Q^{1/2} & 0\\
	0 & R^{1/2}
	\end{bmatrix}\begin{bmatrix}
	{\bPhi}_x\\
	{\bPhi}_u
	\end{bmatrix}\left(I+\begin{bmatrix}
	\Delta_A & \Delta_B
	\end{bmatrix}\begin{bmatrix}
	{\bPhi}_x\\
	{\bPhi}_u
	\end{bmatrix}\right)^{-1}\right\|_{\H2}\\
	\text{s.t.} &\begin{bmatrix}
	zI-\hat{A}+\Delta_A & -\hat{B}+\Delta_B
	\end{bmatrix}\begin{bmatrix}
	\bPhi_x\\
	\bPhi_u
	\end{bmatrix} = I,\\
	& \bPhi_x\in\frac{1}{z}\mathcal{R}\mathcal{H}_{\infty}\cap\cC_x,\quad \bPhi_u\in\frac{1}{z}\mathcal{R}\mathcal{H}_{\infty}\cap\cC_u,
	\end{align}
	The above optimization seeks to find a stabilizing distributed controller that minimizes the worst-case performance achieved on the true system, given the estimates $(\hat{A},\hat{B})$, and the estimation error $\epsilon$. Clearly, this problem is equivalent to its oracle analog if $\epsilon = 0$. However, notice that the above optimization is infinite-dimensional and non-convex in its current form. 
	To deal with its non-convexity,~\cite{dean2017sample} introduces the following surrogate:
	\begin{align}\label{eq2}
	\min_{\gamma\in[0,1)}&\frac{1}{1-r}\min_{\bPhi_x, \bPhi_x}\left\|\begin{bmatrix}
	Q^{1/2} & 0\\
	0 & R^{1/2}
	\end{bmatrix}\begin{bmatrix}
	\bPhi_x\\
	\bPhi_u
	\end{bmatrix}\right\|_{\H2}\\
	\text{s.t.} & \begin{bmatrix}
	zI-\hat{A} & -\hat{B}
	\end{bmatrix}\begin{bmatrix}
	\bPhi_x\\
	\bPhi_u
	\end{bmatrix} = I,\,
	\left\|\begin{bmatrix}
	\frac{\epsilon_A}{\sqrt{\alpha}}\bPhi_x\\
	\frac{\epsilon_B}{\sqrt{1-\alpha}}\bPhi_u
	\end{bmatrix}\right\|_{\Hinf}\leq \gamma,\label{robust}\\
	& \begin{bmatrix}
	\bPhi_x\\
	\bPhi_u
	\end{bmatrix} \in\frac{1}{z}\mathcal{R}\mathcal{H}_{\infty}\cap\cC
	\end{align}
	where $\epsilon_A = \|\hat{A}-\trueA\|_2$ and $\epsilon_B = \|\hat{B}-\trueB\|_2$. It can be easily verified that the above optimization is jointly quasi-convex in $\gamma$ and $(\bPhi_x, \bPhi_u)$. Therefore, upon restricting $(\bPhi_x, \bPhi_u)$ to FIR responses, it can be solved in polynomial time to an arbitrary accuracy. In the absence of sparsity constraints,~\cite{dean2017sample} shows that the above problem gives rise to a robust controller that stabilizes the true system for sufficiently small $\epsilon_A$ and $\epsilon_B$. Moreover,~\cite{dean2017sample} characterizes the gap between the cost of the derived and optimal LQR controllers, and shows that the gap scales as $O(\epsilon_A + \epsilon_B)$. However, care must be taken when extending this approach to the distributed setting:
	
	\noindent{\it 1. Sparsity constraints:} The derived bound on the performance of the synthesized controller in~\cite{dean2017sample} is only valid if there are no sparsity constraints on the system responses. 
	
	\noindent{\it 2. Computational complexity:} As mentioned before, the above optimization is infinite dimensional and hence, intractable to solve. With the goal of reducing~\eqref{eq2} to a finite-dimensional problem,~\cite{dean2017sample} proposes to restrict $(\Phi_x,\Phi_u)$ to FIR responses with length $L$. With this assumption,~\cite{dean2017sample} shows that for a fixed $\gamma$, the inner optimization in~\eqref{robust} can be represented as a semidefinite programming (SDP) with the size $L(n+m)+n$. Moreover,~\cite{dean2017sample} introduces a gridding method to search for the optimal value of $\gamma$ over the interval $[0,1)$. Considering the expensive computational complexity of the available SDP solvers,~\eqref{eq2} quickly becomes prohibitive to solve as the system dimension and/or the length of the FIR responses grow. In particular, using an interior point method~\cite{nesterov1994interior} to solve the inner SDP for every $\gamma$, the proposed algorithm in~\cite{dean2017sample} has the time complexity $\mathcal{O}\left(\left(L(n\!+\!m)\!\right)^{6.5} \frac{1}{\eta}\log\left(\frac{1}{\eta}\right)\right)$ to obtain an $\eta$-accurate solution.
	
	
	\noindent{\it 3. Sample complexity:} Combined with the proposed least-squares estimation method in~\cite{dean2017sample}, the minimum number of sample trajectories to accurately estimate the system matrices scales linearly in the system dimension. This linear dependency makes the accurate estimation impractical, if not impossible, as the system size scales up---this is because no \emph{a priori} knowledge of sparsity in the underlying system is exploited.
	
	In this paper, we will remedy all of the aforementioned issues by introducing a scalable surrogate to the robust optimization problem~\eqref{eq1} with provable optimality guarantees.
	\subsection{Tractable Surrogates}
	We now show how the underlying sparse structure of the system matrices $(\trueA,\trueB)$ and distributed controller can be exploited to develop a tractable and scalable convex surrogate to optimization problem \eqref{eq1}.
	
	Consider the sequence $\cC_v := \{\cC_v(\tau)\}_{\tau = 1}^\infty$, where $\cC_v(0) = \{X|X\in\mathcal{S}(I_n)\}$ and 
	\begin{align}\label{Cv}
	\cC_v(\tau) = \{X_1X_2 + X_3X_4| X_1\in\mathcal{S}(\hat{A}), X_2\in \cC_x(\tau), X_3\in\mathcal{S}(\hat{B}), X_4\in \cC_u(\tau)\}
	\end{align}
	for every $\tau = 1,\dots,\infty$. Assuming that $(\hat{A},\hat{B})$ and $(\trueA,\trueB)$ share the same sparsity pattern, consider the following optimization problem:
	
	\begin{subequations}\label{opt_f}
		\begin{align}
		\min_{\gamma\in[0,1)}\frac{1}{1-\gamma}\underset{\begin{subarray}	.V(0:L)\\
			\Phi_x(1:L)\\ 
			\Phi_u(1:L)
			\end{subarray}}{\min}&\sqrt{\sum_{t=1}^{L}\left\|\begin{bmatrix}
			Q^{1/2} & 0\\
			0 & R^{1/2}
			\end{bmatrix}\begin{bmatrix}
			\Phi_x(t)\\
			\Phi_u(t)
			\end{bmatrix}\right\|^2_{F}}\\
		\mathrm{s.t.}\ &\Phi_x(1) = I+V(0)\label{opt_f_1}\\
		&\Phi_x(t+1) = \hat{A}\Phi_x(t)+\hat{B}\Phi_u(t)+V(t) && t = 1,\dots,L-1\label{opt_f_2}\\
		&0 = \hat{A}\Phi_x(L)+\hat{B}\Phi_u(L)+V(L)\label{opt_f_3}\\
		&\sum_{t=1}^{L}\left\|\begin{bmatrix}
		\bar\epsilon\Phi_x(t)\\
		\bar\epsilon\Phi_u(t)
		\end{bmatrix}_{:,j}\right\|_1\leq\alpha k_{\phi}^{-1/2}\gamma && j=1,\dots,n\label{opt_f_4}\\
		&\sum_{t=0}^{L}\left\|V_{:,j}(t)\right\|_1\leq({1-\alpha}) k_v^{-1}\gamma && j=1,\dots,n\label{opt_f_5}\\
		&\Phi_x(t)\in\cC_x(t),\quad\Phi_u(t)\in\cC_u(t)  && t = 1,\dots,L\label{opt_f_6}\\
		& V(t)\in\cC_v(t) && t = 0,\dots,L\label{opt_f_7}
		\end{align}
	\end{subequations}
	Here, $\alpha\in(0,1)$ is a parameter to be tuned. Furthermore, $\bar{\epsilon}$ is an upper bound on the spectral norm of the true estimation error $\epsilon$, i.e., $\bar{\epsilon}\geq \epsilon$. Later, we will show how to obtain such upper bound directly from the sample trajectories via bootstrapping. The scalar $k_{\phi}$ corresponds to the maximum number of nonzero elements in different rows and columns of $\begin{bmatrix}
	\bPhi_x^\top & \bPhi_u^\top
	\end{bmatrix}^\top$. Similarly, $k_v$ denotes the maximum number of nonzero elements in different rows and columns of $\bV$; we will explain later how to obtain $k_v$ based on the imposed sparsity patterns of the system responses. 
	Let a globally optimal solution of the above optimization be denoted by $(\bPhi_x^L, \bPhi_u^L, \bV^L,\gamma^L)$. The inner optimization problem of~\eqref{opt_f} can be written as a parametric QP with respect to $\gamma$ and is denoted by $\mathrm{OPT}(\gamma)$, whose optimal objective value is referred to as $g(\gamma)$. It is easy to see that $g(\gamma)$ is defined over the domain $[\gamma_0,+\infty)$ for some $\gamma_0\geq0$, and is monotonically decreasing.
	
	\begin{sloppypar}
			We will discuss a number of key properties of this problem. First, notice that the optimization is over only the first $L$ components of the system responses, thus yielding a finite-dimensional approximation of the previous infinite-dimensional problem.  The slack variables $V(0),V(1),\dots, V(L)$ are used to capture the error incurred by this truncation. In Theorem~\ref{thm_FIR}, we show that the approximation error incurred by restricting our optimization to the first $L$ system response elements decays exponentially with respect to $L$. Moreover, as will be shown in Lemma~\ref{lem_Cv}, the supports of the introduced slack variables are only slightly larger than those of the  system responses. Therefore, if the computed system responses are sparse, so are the slack variables. This will in turn help reduce the number of variables in the problem, thereby resulting in a significant computational saving. Finally, a close comparison between~\eqref{opt_f} and~\eqref{eq2} reveals that the constraint imposed on the $\Hinf$-norm of the system responses in the latter is replaced by induced norm-1 constraints on the system response elements and the slack variables. Considering the fact that these constraints can be represented as linear inequalities, we will later show how to efficiently decompose the proposed optimization problem into a series of small and independent QPs.
	\end{sloppypar}

	
	The next lemma characterizes the sparsity structure of the set $\cC_v$. To simplify notation, $k$ will be used to denote the maximum number of nonzero elements of every row and column of $\begin{bmatrix}
	\trueA & \trueB
	\end{bmatrix}$ and feasible $\begin{bmatrix}
	\Phi_x^\top(\tau) & \Phi_u^\top(\tau)
	\end{bmatrix}^\top, \tau = 1,\dots, L$. Furthermore, we will drop the scripts from a time-dependent sequence $\{M(\tau)\}_{\tau=t_1}^{t_2}$ whenever they are implied by the context.
	\begin{lemma}\label{lem_Cv}
		The following statements hold:
		\begin{itemize}
			\item[1.] The maximum number of nonzero elements in the rows or columns of every ${M}\in\cC_v$ is upper bounded by $2k^2$.
			\item[2.] The equality $\cC_v(\tau) = \mathcal{S}(P_1P_2+P_3P_4)$ is satisfied for every $\tau = 1,\dots, L$, where $P_1 = \mathrm{supp}(\hat{A})$ and $P_3 = \mathrm{supp}(\hat{B})$. Furthermore, $P_2$ and $P_4$ are binary matrices with the maximum number of nonzero elements that satisfy $P_2\in C_x(\tau)$ and $P_4\in C_u(\tau)$. 
		\end{itemize}
	\end{lemma}
	\begin{proof}
		The proofs of both statements are immediately implied by the sparsity patterns of $\hat{A}$, $\hat{B}$, and the elements of $\cC_x(\tau)$ and $\cC_u(\tau)$.
	\end{proof}
	Since $P_1$, $P_2$, $P_3$, and $P_4$ are sparse matrices, Lemma~\ref{lem_Cv} implies that $\{\cC_v(\tau)\}$ can be efficiently characterized by sparse matrix multiplication and summation. 
	
	\subsection{Optimality gap} In this subsection, we analyze the performance of the controller derived from~\eqref{opt_f}. 
	The following is the first main theorem of the paper.
	
	
	\begin{theorem}\label{thm_FIR}
		Let $J_\star$ be the oracle cost and $(\gamma^L,\bPhi_x^L,\bPhi_u^L)$ be the optimal solution of~\eqref{opt_f}. Suppose that $\hat{A}$ and $\hat{B}$ have the same sparsity structure as $\trueA$ and $\trueB$, and that
		\begin{align}\label{cond}
		\bar\epsilon<\frac{(1-\rho_\star)\min\{\alpha,1-\alpha\}}{32C_\star\rho_\star} k^{-2}, \quad L> \frac{2\log(k)+\log\left(\frac{4\sqrt{2}(\|\trueA\|_\infty+\|\trueB\|_\infty)}{1-\alpha}\right)}{1-\rho_\star}.
		\end{align} 
		Then, the following statements hold:
		\begin{itemize}
			\item[1.] ${\bK}^L = \bPhi^L_u{\bPhi^L_x}^{-1}$ stabilizes the true system.
			\item[2.] We have
			\begin{align}\label{opt_gap}
			\hspace{-1cm}\frac{J(A,B,{\bK}^L)-J_\star}{J_\star}\leq\underbrace{\frac{16}{\min\{\alpha,1-\alpha\}}\frac{C_\star\rho_\star}{(1-\rho_\star)}k^{2}\bar{\epsilon}}_{\text{uncertainty error}}+\underbrace{\frac{2\sqrt{2}}{1-\alpha}(\|\trueA\|_\infty+\|\trueB\|_\infty)C_\star k^2\rho_\star^L}_{\text{truncation error}}
			\end{align}
		\end{itemize}
		
	\end{theorem}
\begin{proof}
	See Appendix~\ref{p_thm_FIR}.
\end{proof}
	
	Theorem \ref{thm_FIR} quantifies the effects of model uncertainty and spatiotemporal truncation on the optimality gap of the designed distributed controller. In particular, it shows that the uncertainty error is a linear function of $\bar\epsilon$, which is an available upper bound on the actual estimation error. On the other hand, even with $\bar\epsilon = \epsilon = 0$, one cannot guarantee a zero optimality gap for the designed controller due to the error incurred by the truncation of the system responses. Theorem~\ref{thm_FIR} together with the fact that $0\leq\rho_\star<1$ implies that this truncation error decreases exponentially fast with respect to the FIR length $L$.  Further, the smaller $\rho_\star$ is, i.e., the faster the optimal system response decays to zero, the faster the truncation error decays. 
	Finally, if we assume that $\|\trueA\|_\infty$, $\|\trueB\|_\infty$, $C_\star$, and $\rho_\star$ do not scale with the system dimensions, then the derived bounds show that the uncertainty and truncation errors are independent of the system dimension and instead, they only scale with the number of nonzero elements in different rows or columns of the system matrices and responses. Note that $\|\trueA\|_\infty$, $\|\trueB\|_\infty$, $C_\star$, and $\rho_\star$ are defined in terms of the element-wise norm of the system matrices and responses; indeed, the assumption on independence of these quantities from the system dimension are milder and more practical than similar assumptions on their spectral norms, as is usually done in the literature.
	
	\subsection{Sample complexity}\label{sec:sysID}
	\begin{sloppypar}
		Recently, special attention has been devoted to estimating state-space parameters of linear and time-invariant systems based on a limited number of input-output sample trajectories, defined as sequences $\{(x^{(i)}(\tau),u^{(i)}(\tau))\}_{\tau = 0}^T$ with $i = 1,2,...,d$, where $d$ is the number of available sample trajectories and $T$ is the length of each sample trajectory. To simplify notation, the superscript $i$ is dropped from the sample trajectories when $d=1$. In general, there are two different approaches to the identification of state-space parameters in the full observation setting:
	\end{sloppypar}

	\vspace{2mm}
	\noindent\textit{Single sample trajectory:} In this method, the system identification is performed based on a single sample trajectory and the sample complexity of the proposed estimator is characterized in terms of the horizon length $T$ (also referred to as learning time) over which state-input pairs $\{(x(\tau),u(\tau))\}_{\tau = 0}^T$ are collected. This approach is most suitable when the open-loop system is stable, or if an initial stabilizing controller is provided. Notice that the assumption on stability is necessary, as the ordinary least-squares estimator may not be consistent if the system has unstable modes~\cite{sarkar2018fast}. From a practical perspective, system instability may also impose stringent limits on the learning time to ensure system safety, thereby restricting the collected sample size.
	
	\vspace{2mm}
	\noindent\textit{Multiple sample trajectories:} In this approach, the learning time is fixed and instead, the number of sample trajectories is chosen to be sufficiently large. This approach is equally applicable to stable and unstable systems; however, one needs to be able to reset the system at the start of each sample trajectory, which may not be possible in practice.
	
	As we seek \emph{sparse} state-space parameters $(\hat A, \hat B)$, we draw upon techniques from the structured inference literature---specifically the well-known Lasso estimator \cite{tibshirani1996regression,donoho2006compressed}---and, given a prescribed regularization coefficient $\lambda$, consider the following $M$-estimator:
	\begin{align}\label{lasso}
	&(\hat{A},\hat{B})\\ &= \arg\min_{A, B}\frac{1}{2(t_2\!-\!t_1)d}\sum_{i=1}^{d}\sum_{t=t_1}^{t_2} \left\|x^{(i)}(t+1)\!-\!\left(Ax^{(i)}(t)\!+\!Bu^{(i)}(t)\right)\right\|_2^2\!+\! \lambda(\|A\|_1+\|B\|_1)\nonumber
	\end{align}
	which is referred to as \texttt{LASSO}$(1:d,t_1:t_2)$ in the sequel. For simplicity of notation, let $\hat\Psi = \begin{bmatrix}
	\hat{A} & \hat{B}
	\end{bmatrix}^\top$ and $\Psi_\star = \begin{bmatrix}
	{A}_\star & {B}_\star
	\end{bmatrix}^\top$ denote the estimated and true system matrices, respectively. In~\cite{Salar18,fattahi2019learning}, variants of the regression problem~\eqref{lasso} are used to address the problem of sparse system identification with single and multiple sample trajectories. 
	
	\begin{remark}
		{\it As mentioned before, the system identification based on a single trajectory relies on the availability of an initial distributed controller $K_0$. Such initial controller may not be necessary if the system is internally stable or it may be obtained based on domain knowledge. Alternatively, one can use the system identification technique developed in~\cite{Salar18} that is based on multiple sample trajectories and hence, bypass the need for such initial controllers. Indeed, our optimization technique can be readily combined with the results of~\cite{Salar18} to obtain end-to-end bounds on the sample complexity of the designed distributed controller based on multiple sample trajectories. Due to space restrictions and similarity of the results, we only focus on the system identification with single sample trajectory in this paper.}
	\end{remark}
	
	
	
	Assume that $w(t)\overset{iid}{\sim}\mathcal{N}(0,\sigma_w^2I)$ for some $\sigma_w>0$ and the system is equipped with a known stabilizing and static localized controller $K_0$ with a sparse structure. As mentioned before, $K_0$ can be set to zero if the system is internally stable. Furthermore, suppose that $u(t) = K_0x(t)+v(t)$ with $v(t)\overset{iid}{\sim}\mathcal{N}(0,\sigma_v^2I)$ for some $\sigma_v>0$. 
	
	Upon the stability of $A+BK_0$, the vector $\begin{bmatrix}
	x(t)^\top & u(t)^\top
	\end{bmatrix}^\top$ converges to a stationary distribution $\mathcal{N}(0,M_\star)$, where $M_\star$ is defined as
	\begin{equation}\label{true_M}
	{M}_\star = \begin{bmatrix}
	P & PK_0^\top\\
	K_0P & K_0PK_0^\top + \sigma^2_v I 
	\end{bmatrix}
	\end{equation}
	and $P$ satisfies the following Lyapunov equation:
	\begin{align}\label{true_P}
	(\trueA+\trueB K_0)P(\trueA+\trueB K_0)^\top-P+\sigma^2_wI + \sigma_v^2\trueB \trueB^\top = 0
	\end{align}
	We assume that the initial state rests at its stationary distribution. As explained in~\cite{fattahi2019learning}, this assumption is mild since the state vector converges to its stationary distribution exponentially fast. 
	\begin{assumption}\label{ass1}
		The following statements hold:
		\begin{align}
		&\max_{1\leq j\leq n}\left\{\max_{i\in\mathcal{A}^c_j}\left\{\left\|{{M_\star}_{i \mathcal{A}_j}({M_\star}_{\mathcal{A}_j \mathcal{A}_j})^{-1}}\right\|_{1}\right\}\right\}\leq 1-r, && \min_{1\leq j\leq n}\lambda_{\min}({M_\star}_{\mathcal{A}_j \mathcal{A}_j})\geq C_{\min},\nonumber\\
		&\max_{1\leq j\leq n}\|({M_\star}_{\mathcal{A}_j \mathcal{A}_j})^{-1}\|_\infty\leq D_{\max}, && \min_{1\leq j\leq n}\left\{\max_{i\in\mathcal{A}_j}\left\{|{\Psi_\star}_{ij}|\right\}\right\} \geq \Psi_{\min}\nonumber
		\end{align}
		for some constants $0<r<1$, $1\geq C_{\min}>0$, $D_{\max}\geq1$ and $1\geq \Psi_{\min}>0$.
	\end{assumption}
	Roughly speaking, the above assumptions imply that $M_\star$ should satisfy a certain level of \textit{mutual incoherency}, and it should possess uniformly bounded norms---such assumptions are essential for the correct sparsity recovery of $\Psi_\star$~\cite{elad2010sparse, fattahi2019learning, wainwright2009sharp}. Furthermore, there should be a non-vanishing gap between the zero and nonzero elements of $\Psi_\star$. the We refer the reader to~\cite{fattahi2019learning} for an extensive discussion on the practical implications of the above conditions.
	The following proposition determines the non-asymptotic estimation error of \texttt{LASSO}$(1,1:T-1)$.
	\begin{proposition}[\cite{fattahi2019learning}]\label{prop_single}
		Suppose that $k\geq 2$ and the following conditions hold:
		\begin{align}\label{lambda_T}
		\lambda = \mathcal{C}_{ \mathrm{s};\lambda}\sqrt{\frac{\log((n+m)/\delta)}{T}},\qquad T\geq \mathcal{C}_{ \mathrm{s};T} k^2\log((n+m)/\delta),
		\end{align}
		Then, under Assumption~\ref{ass1}, $\texttt{LASSO}(1,1:T-1)$ recovers the true sparsity pattern of $\Psi_\star$ and it incurs the element-wise estimation error
		\begin{align}\label{est_err}
		\|\hat\Psi-\Psi_\star\|_{\infty}\leq \mathcal{C}_{\mathrm{s};\mathrm{err}}\sqrt{\frac{\log((n+m)/\delta)}{T}}
		\end{align}
		with probability at least $1-\delta$.
	\end{proposition}
	The system complexity constants $\mathcal{C}_{ \mathrm{s};\lambda}$, $\mathcal{C}_{ \mathrm{s};T}$, and $\mathcal{C}_{ \mathrm{s};\mathrm{err}}$ depend on the spectral radius of the closed-loop gain $A+BK_0$, as well as the parameters $r$, $C_{\min}$, $D_{\max}$, and $\Psi_{\min}$. See~\cite{fattahi2019learning} for the exact definitions of these complexity constants.
	
	Equipped with this proposition and Theorem~\ref{thm_FIR}, we present the following theorem that characterizes the sample complexity of the derived distributed controller in terms of the learning time and the FIR lengths of the system responses.
	
	\begin{theorem}\label{thm_bound}
		\begin{sloppypar}
			Suppose that $k\geq 2$, Assumption~\ref{ass1} holds, and $\texttt{LASSO}(1,1:T-1)$ is used to obtain the estimates $(\hat{A},\hat{B})$. Furthermore, suppose that $\bar{\epsilon} = \zeta \mathcal{C}_{\mathrm{s};\mathrm{err}}\sqrt{\frac{k^2\log((n+m)/\delta)}{T}}$ for an arbitrary $\zeta\geq 1$ and that 
		\end{sloppypar}
		\begin{align}
		\lambda &= \mathcal{C}_{ \mathrm{s};\lambda}\sqrt{\frac{\log((n+m)/\delta)}{T}},\label{Lambda}\\
		T&\geq \max\left\{\left(\frac{32}{\min\{\alpha,1-\alpha\}}\frac{C_\star\rho_\star\zeta \mathcal{C}_{ \mathrm{s};\mathrm{err}}}{1-\rho_\star}\right)^2k^6,\mathcal{C}_{ \mathrm{s};T}k^2\right\}\log((n+m)/\delta),\label{T}\\
		L&\geq \frac{2\log(k)+\log\left(\frac{4\sqrt{2}(\|\trueA\|_\infty+\|\trueB\|_\infty)}{1-\alpha}\right)}{1-\rho_\star}\label{L}.
		\end{align}
		where $\alpha\in(0,1)$ is an arbitrary and predefined parameter in~\eqref{opt_f}. Then, the following statements hold with probability at least $1-\delta$:
		\begin{itemize}
			\item[1.] ${\bK}^L = \bPhi^L_u{\bPhi^L_x}^{-1}$ stabilizes the true system.
			\item[2.] We have
			\begin{align}\nonumber
			\hspace{-1cm}\frac{J(A,B,{\bK}^L)-J_\star}{J_\star}\leq&{\frac{16}{\min\{\alpha,1-\alpha\}}\frac{C_\star\rho_\star\zeta \mathcal{C}_{\mathrm{s};\mathrm{err}}}{1-\rho_\star}k^3\sqrt{\frac{\log((n+m)/\delta)}{T}}}\nonumber\\
			&+{\frac{2\sqrt{2}}{1-\alpha}(\|\trueA\|_\infty+\|\trueB\|_\infty)C_\star k\rho_\star^L}\nonumber
			\end{align}
		\end{itemize}
	\end{theorem}
	
	\begin{proof}
		Theorem~\ref{thm_FIR} and Proposition~\ref{prop_single} can be used to prove this theorem. First, note that~\eqref{Lambda} and~\eqref{T} guarantee the validity of~\eqref{lambda_T}. Therefore, $\texttt{LASSO}(1,1:T-1)$ can recover the correct sparsity pattern of the system matrices and the estimation error bound~\eqref{est_err} holds with probability of at least $1-\delta$. This implies that
		\begin{align}
		\epsilon = \|\hat\Psi-\Psi_\star\|_{2}\leq k\|\hat\Psi-\Psi_\star\|_{\infty}\leq\zeta \mathcal{C}_{\mathrm{s};\mathrm{err}}\sqrt{\frac{k^2\log((n+m)/\delta)}{T}}=\bar{\epsilon}
		\end{align}
		with the same probability.
		Combined with~\eqref{T} and~\eqref{L}, this certifies the validity of~\eqref{cond}. Therefore,~\eqref{opt_gap} holds with probability of at least $1-\delta$. Replacing $\bar{\epsilon}$ with  $\zeta\mathcal{C}_{\mathrm{s};\mathrm{err}}k\sqrt{\frac{\log((n+m)/\delta)}{T}}$ in~\eqref{opt_gap} completes the proof.
	\end{proof}
	
	Under the assumption that $C_\star$, $\rho_*$, $\|\trueA\|_\infty$, $\|\trueA\|_\infty$, and the system complexity constants do not scale with the system dimension, Theorem~\ref{thm_bound} implies that $T = \Omega(k^6\log(n+m))$ is enough to guarantee that the optimality gap of the designed controller is on the order of $\mathcal{O}(k^3\sqrt{\log(n+m)/T}+k\rho_\star^L)$. Assuming that the dynamics and controller have sparse structures, i.e., $k\ll n+m$, the proposed bound improves upon the existing sample complexity bounds for learning optimal LQR controllers which scale linearly with the system dimension~\cite{dean2017sample, dean2018regret}.
	
	\begin{remark}
		While the proposed method is best suited for designing controllers with sparse system responses, its performance can be compared against a more general oracle optimization~\eqref{eq22}, where the constraint sets $\cC_x$ and $\cC_u$ are relaxed to \textbf{weakly sparse} structures. Under such circumstances, an optimal LQR controller can be a valid oracle controller, provided that its induced system responses are weakly sparse or, equivalently, they have spatially decaying structures; see~\cite{motee2008optimal, motee2014sparsity, virtual}. Even though such generalizations are not discussed in this paper, we note that the derived sub-optimality gap of the designed controller in Theorems~\ref{thm_bound} and~\ref{thm_FIR} can be extended to this setting, with an additional non-vanishing term capturing the \textbf{model selection error}. 
	\end{remark}
	\subsection{Computational complexity}
	\begin{sloppypar}
		In this subsection, we propose an efficient algorithm for solving~\eqref{opt_f}. It is easy to verify that the proposed optimization problem is jointly quasiconvex. In particular, it is convex with respect to $\left( \left\{{\Phi}_x(t)\right\}, \left\{{\Phi}_u(t)\right\}, \left\{V(t)\right\}\right)$ (after fixing $\gamma$) and quasiconvex with respect to $\gamma$ (after fixing $\left( \left\{{\Phi}_x(t)\right\}, \left\{{\Phi}_u(t)\right\}, \left\{V(t)\right\}\right)$). 
	\end{sloppypar}
	\begin{lemma}\label{l_unique}
		For every fixed and feasible $\bar\gamma$, $\mathrm{OPT}(\bar\gamma)$ has a unique solution.
	\end{lemma}
	\begin{proof}
		\begin{sloppypar}
			Notice that $\{V(t)\}$ can be uniquely written in terms of $\left\{{\Phi}_x(t)\right\}$ and $\left\{{\Phi}_u(t)\right\}$. This, together with the fact that the objective is strictly convex, results in the uniqueness of the solution.
		\end{sloppypar}
	\end{proof}
	
	Lemma~\ref{l_unique} and the quasiconvexity of $g(\gamma)$ do not necessarily result in the uniqueness of the solution for~\eqref{opt_f} since $g(\gamma)$ may contain spurious local minima in its \textit{flat regions}. A naive approach to circumvent this issue is to discretize $\gamma$ within the interval $[0,1)$ with the points $\{\gamma_1,\dots,\gamma_N\}$, compute $g(\gamma_i)$ for every $1\leq i\leq N$, and select the solution with the lowest cost. However, notice that in this approach, the number of discrete points has undesirable dependency on the required accuracy of the solution: roughly speaking, one needs to evaluate and optimize over $\Omega(1/\epsilon)$ discrete points in order to get a solution whose cost is $\epsilon$-away from the optimal cost. In the next proposition, we show that~\eqref{opt_f} is in fact unimodal with respect to $\gamma$ and hence, it is free of spurious local minima (i.e. non-global local minima).\footnote{Note that another approach for eliminating the spurious local minima in the flat regions of a quasiconvex optimization problem is a reformulation based on its sublevel sets; see~\cite{boyd2004convex}. However, this method will destroy the decomposibility of~\eqref{opt_f}; a feature that is at the core of near-linear solvability of~\eqref{opt_f}, as will be shown later in the paper.} 
	The unimodal property of~\eqref{opt_f} with respect to $\gamma$ implies that a simple application of the golden-section search method on $\gamma$ can find an $\epsilon$-accurate solution by computing $g(\gamma_i)$ at no more than $O(\log(1/\epsilon))$ points. 
	
	\begin{proposition}\label{prop_unimodal}
		Suppose that~\eqref{opt_f} is feasible. Furthermore, suppose that $\gamma_0$ is the smallest value such that $0\leq\gamma_0<1$ and $\mathrm{OPT}(\gamma_0)$ is feasible. Then, $\frac{g(\gamma)}{1-\gamma}$ is unimodal in the interval $[\gamma_0,1)$.
	\end{proposition}

\begin{proof}
	See Appendix~\ref{p_prop_unimodal}.
\end{proof}

	For a fixed $\gamma$, problem $\mathrm{OPT}(\gamma)$ can be decomposed into $n$ parallel sub-problems over the columns of
	\vspace{-2mm}
	\begin{align}\label{mat}
	\begin{bmatrix}
	\Phi_x(1)^\top & \dots & \Phi_x(L)^\top & \Phi_u(1)^\top & \dots & \Phi_u(L)^\top & V(0)^\top &\dots V(L)^\top
	\end{bmatrix}^\top
	\end{align}
	\begin{sloppypar}
		\noindent In particular, define $\mathrm{OPT}_j(\gamma)$ as $\mathrm{OPT}(\gamma)$ after replacing the variable matrices $(\{\Phi_x(t)\},\{\Phi_u(t)\},\{V(t)\})$ with $(\{[\Phi_x(t)]_{:,j}\},\{[\Phi_u(t)]_{:,j}\},\{[V(t)]_{:,j}\})$, as in:
	\end{sloppypar}
	\begin{subequations}\label{opt_fj}
		\begin{align}
		\underset{\begin{subarray}	.\{[V(t)]_{:,j}\}\\
			\{[\Phi_x(t)]_{:,j}\}\\ 
			\{[\Phi_u(t)]_{:,j}\}
			\end{subarray}}{\min}&\sqrt{\sum_{t=1}^{L}\left\|\begin{bmatrix}
			Q^{1/2} & 0\\
			0 & R^{1/2}
			\end{bmatrix}\begin{bmatrix}
			\Phi_x(t)\\
			\Phi_u(t)
			\end{bmatrix}_{:,j}\right\|^2_{F}}\\
		\mathrm{s.t.}\ &[\Phi_x(1)]_{:,j} = I_{:,j}+[V(0)]_{:,j}\label{opt_fj_1}\\
		&[\Phi_x(t+1)]_{:,j} = \hat{A}[\Phi_x(t)]_{:,j}+\hat{B}[\Phi_u(t)]_{:,j}+[V(t)]_{:,j} && t = 1,\dots,L-1\label{opt_fj_2}\\
		&0 = \hat{A}[\Phi_x(L)]_{:,j}+\hat{B}[\Phi_u(L)]_{:,j}+[V(L)]_{:,j}\label{opt_fj_3}\\
		&\sum_{t=1}^{L}\left\|\begin{bmatrix}
		\bar\epsilon\Phi_x(t)\\
		\bar\epsilon\Phi_u(t)
		\end{bmatrix}_{:,j}\right\|_1\leq\alpha k_{\phi}^{-1/2}\gamma &&  t = 1,\dots,L\label{opt_fj_4}\\
		&\sum_{t=0}^{L}\left\|[V(t)]_{:,j}\right\|_1\leq({1-\alpha}) k_v^{-1}\gamma &&  t = 0,\dots,L\label{opt_fj_5}\\
		&[\Phi_x(t)]_{:,j}\in \cC_{x;j}(t),\quad [\Phi_u(t)]_{:,j}\in \cC_{u;j}(t) && t = 1,\dots,L\label{opt_fj_6}\\
		& [V(t)]_{:,j} \in \cC_{v;j}(t) && t = 0,\dots,L\label{opt_fj_7}
		\end{align}
	\end{subequations}
	where $\cC_{x;j}(t) = \{X_{:,j}:X\in\cC_x(t)\}$, $\cC_{u;j}(t) = \{X_{:,j}:X\in\cC_u(t)\}$, and $\cC_{v;j}(t) = \{X_{:,j}:X\in\cC_v(t)\}$. Furthermore, let $g_j(\gamma)$ denote its optimal objective value. Then, $g(\gamma) = \sqrt{\sum_{j=1}^{n}g_j(\gamma)^2}$ and the optimal solution of $\mathrm{OPT}(\gamma)$ can be obtained by replacing the $j^{\text{th}}$ column of~\eqref{mat} with the solution of the sub-problem $\mathrm{OPT}_j(\gamma)$ for every $j = 1,\dots,n$. 
	
	The next lemma shows that the sub-problem $\mathrm{OPT}_j(\gamma)$ can be reformulated as a small QP whose size is independent of $n$.
	\begin{lemma}\label{l_reducedQP}
		The sub-problem $\mathrm{OPT}_j(\gamma)$ can be written as a QP over $O(Lk^2)$ variables subject to $O(Lk^2)$ constraints.
	\end{lemma}
	\begin{proof}
		\begin{sloppypar}
			For every $t=0,\dots,L$, let $(\Phi^{n_j}_x(t), \Phi^{n_j}_u(t), V^{n_j}(t))$ correspond to $(\Phi_x(t), \Phi_u(t), V(t))$ after removing the elements that are set to zero via the sparsity constraints~\eqref{opt_f_6} and~\eqref{opt_f_7}. It is easy to see that $\mathrm{OPT}_j(\gamma)$ can be written in terms of $(\{\Phi^{n_j}_x(t)\}, \{\Phi^{n_j}_u(t)\}, \{V^{n_j}(t)\})$ with a total number of $O(Lk^2)$ variables. The rest of the proof is devoted to show how to reduce the number of constraints in $\mathrm{OPT}_j(\gamma)$ to $O(Lk^2)$. 
			Let $\bPhi^{n_j}_x$, $\bPhi^{n_j}_u$, and $\bV^{n_j}$ denote $\sum_{t=1}^{L}\Phi^{n_j}_x(t)z^{-t}$, $\sum_{t=1}^{L}\Phi^{n_j}_u(t)z^{-t}$, and $\sum_{t=0}^{L}V^{n_j}(t)z^{-t}$, respectively. The constraints~\eqref{opt_fj_1}-\eqref{opt_fj_3} can be written compactly as
		\end{sloppypar}
		\begin{equation}
		{\begin{bmatrix}
			zI-\hat{A} & -\hat{B} & -I
			\end{bmatrix}}\begin{bmatrix}
		\bPhi_x\\
		\bPhi_u\\
		\bV
		\end{bmatrix}_{:,j} = I_{:,j}\iff \mathbf{M}_j\begin{bmatrix}
		\bPhi^{n_j}_x\\
		\bPhi^{n_j}_u\\
		\bV^{n_j}
		\end{bmatrix} = I_{:,j}
		\end{equation} 
		\begin{sloppypar}
			Here, $\mathbf{M}_j$ is equal to $\sum_{t=0}^LM_j(t)z^{-t}$, where $M_j(t)$ is defined as $\begin{bmatrix}
			zI-\hat{A} & -\hat{B} & -I
			\end{bmatrix}$, after removing the columns that correspond to the zero elements of $\begin{bmatrix}
			\Phi_x(t)^\top & \Phi_u(t)^\top & V(t)^\top
			\end{bmatrix}_{j,:}^\top$ enforced by the sparsity constraints. The matrix $\mathbf{M}_j$ has at most $n$ rows and $2k^2+k$ columns. On the other hand, every column of $\begin{bmatrix}
			zI-\hat{A} & -\hat{B}
			\end{bmatrix}$ has at most $k+1$ number of nonzero elements. Similarly, every column of $-I$ has exactly one nonzero element. Therefore, a simple calculation yields that $\mathbf{M}_j$ can have at most $3k^2+k$ number of nonzero rows. This together with the definition of $\mathbf{M}_j$ implies that~\eqref{opt_fj_1}-\eqref{opt_fj_3} can be reduced to $O(Lk^2)$ linear constraints. Finally,~\eqref{opt_fj_4} and~\eqref{opt_fj_5} can be trivially written as a set of $O(Lk^2)$ linear inequalities by introducing $O(Lk^2)$ slack variables. This completes the proof.
		\end{sloppypar}
	\end{proof}

It is worthwhile to mention that the above lemma is a generalization to the dimension reduction algorithm introduced in~\cite{wang2017separable}.
	
	\begin{remark}
		{\it Note that for every index $j$, the aforementioned reduced QP can be efficiently constructed in an offline fashion before running Algorithm~\ref{alg:solve} detailed below, provided that the estimated system matrices $(\hat{A},\hat{B})$ and the sparsity constraints~\eqref{opt_fj_6} and~\eqref{opt_fj_7} are given in sparse matrix formats, such as Coordinate list~\cite{golub2012matrix}. While we do not discuss the structure of such representations, we note that the complexity of constructing these reduced QPs is dominated by that of Algorithm~\ref{alg:solve}.} 
	\end{remark}
	
	\begin{remark}
		Without loss of generality, we assume that the proposed optimization~\eqref{opt_f} is finitely-representable on a Turing machine. In other words, the total number of digits required to write (or accurately approximate) the input data for~\eqref{opt_f} is a finite number $D$. This is a common assumption made for the complexity analysis of optimization problems; see e.g.~\cite{Vavasis2001}.
	\end{remark}

	\begin{definition}
		An algorithm solves an optimization problem that is finitely-representable on a Turing machine to $\eta$-accuracy if the following statements hold:
		\begin{itemize}
			\item[-] It returns a feasible solution if and only if the problem is feasible,
			\item[-] Upon feasibility, it returns a feasible solution whose objective value is greater than the optimal objective value by no more than $\eta$.
		\end{itemize}
	\end{definition}
	
	\begin{sloppypar}
		Algorithm~\ref{alg:solve} delineates the proposed method for solving~\eqref{opt_f}. In particular, it uses a golden-section search method to optimize over the scalar variable $\gamma$, while solving multiple small QPs at each iteration to obtain $g(\gamma)$. At any iteration, $g(\gamma)$ is set to $+\infty$ if at least one of $\mathrm{OPT}_1(\gamma),\dots,\mathrm{OPT}_n(\gamma)$ is infeasible. Suppose $g(\gamma)$ has the domain $[\gamma_0,+\infty)$ for some $\gamma_0\geq 0$. It is easy to verify that a finite value for $\gamma_0$ always exists; however, $\gamma_0<1$ is required for~\eqref{opt_f} to be feasible.
	\end{sloppypar}
	\begin{algorithm}
		\caption{}
		\label{alg:solve}
		\begin{algorithmic}[1]
			\STATE{{\bf input:} Estimates $\hat{A}$, $\hat{B}$, estimation error $\bar{\epsilon}$, and accuracy parameters $\eta_1$, and $\eta_2$}
			\STATE{{\bf output:} $\{\Phi_x(t)\}$, $\{\Phi_u(t)\}$, $\{V(t)\}$, and $g(\gamma)$}
			\STATE{obtain $g(1)$ by solving $n$ sub-problems $\mathrm{OPT}_1(1),\dots,\mathrm{OPT}_n(1)$ to $\frac{\eta_2}{n}$-accuracy using interior point method.}
			\IF{$g(1) = +\infty$}
			\RETURN{Infeasible}
			\ELSE
			\STATE{{\bf set} $\gamma_a \leftarrow 0$, $\gamma_b \leftarrow 1$, $\gamma_c \leftarrow 1-\frac{2}{1+\sqrt{5}}$, and $\gamma_d \leftarrow \frac{2}{1+\sqrt{5}}$}
			\WHILE{$|\gamma_b-\gamma_a|>\eta_1$}
			\STATE{Solve $\mathrm{OPT}(\gamma_c)$ by solving $n$ sub-problems $\mathrm{OPT}_1(\gamma_c),\dots,\mathrm{OPT}_n(\gamma_c)$ to $\frac{\eta_2}{n}$-accuracy using interior point method. Let the corresponding objective value be denoted as $g_{\mathrm{ap}}(\gamma_c)$.}
			\STATE{Solve $\mathrm{OPT}(\gamma_d)$ by solving $n$ sub-problems $\mathrm{OPT}_1(\gamma_d),\dots,\mathrm{OPT}_n(\gamma_d)$ to $\frac{\eta_2}{n}$-accuracy using interior point method. Let the corresponding objective value be denoted as $g_{\mathrm{ap}}(\gamma_d)$.}
			\IF{$\frac{g_{\mathrm{ap}}(\gamma_c)}{1-\gamma_c}<\frac{g_{\mathrm{ap}}(\gamma_d)}{1-\gamma_d}$}
			\STATE{{\bf set} $\gamma_b \leftarrow \gamma_d$}
			\ELSE{}
			\STATE{{\bf set} $\gamma_a \leftarrow \gamma_c$}
			\ENDIF
			\STATE{$\gamma_c \leftarrow \gamma_b-\frac{2}{1+\sqrt{5}}(\gamma_b-\gamma_a)$ and $\gamma_d \leftarrow \gamma_a+\frac{2}{1+\sqrt{5}}(\gamma_b-\gamma_a)$}
			\ENDWHILE
			\STATE{$\bar\gamma\leftarrow(\gamma_a+\gamma_b)/2$}
			\STATE{obtain $(\{\bar\Phi_x(t)\},\{\bar\Phi_u(t)\},\{\bar V(t)\}, g(\bar\gamma))$ by solving $n$ sub-problems $\mathrm{OPT}_1(\bar\gamma),\dots,\mathrm{OPT}_n(\bar\gamma)$ to $\frac{\eta_2}{n}$-accuracy using interior point method. Let the corresponding objective value be denoted as $g_{\mathrm{ap}}(\bar\gamma)$.}
			\IF{$g_{\mathrm{ap}}(\bar\gamma)=+\infty$}
			\RETURN{Infeasible}
			\ELSE
			\RETURN{$(\{\bar\Phi_x(t)\},\{\bar\Phi_u(t)\},\{\bar V(t)\}, \bar\gamma)$}
			\ENDIF
			\ENDIF
		\end{algorithmic}
	\end{algorithm}
	
	Define $\underline{t}$ and $\overline{t}$ as the smallest and largest integers such that
	\begin{align}
	\underline{\eta_1} = \left(\frac{2}{1+\sqrt{5}}\right)^{\underline{t}} \leq \eta_1,\ \ \overline{\eta_1} = \left(\frac{2}{1+\sqrt{5}}\right)^{\overline{t}} > \eta_1
	\end{align}
	Furthermore, define 
	\begin{align}\label{const}
	\Delta_\gamma = \left(\frac{4}{1+\sqrt{5}}-1\right)\overline{\eta_1}
	\end{align}
	Let $g_{\mathrm{ap}}(\gamma_c)$ and  $g_{\mathrm{ap}}(\gamma_d)$ denote the objective values of the problems $\mathrm{OPT}(\gamma_c)$ and $\mathrm{OPT}(\gamma_c)$ when they are solved to $\eta_2$-accuracy. At each iteration, Algorithm~\eqref{alg:solve} shrinks the interval $[\gamma_a,\gamma_b]$ by comparing the values of $\frac{g_{\mathrm{ap}}(\gamma_c)}{1-\gamma_c}$ and $\frac{g_{\mathrm{ap}}(\gamma_d)}{1-\gamma_d}$, while ensuring that $\gamma^L\in [\gamma_a,\gamma_b]$. However, notice that $g_{\mathrm{ap}}(\gamma_c)$ and $g_{\mathrm{ap}}(\gamma_d)$ are the approximations of $g(\gamma_c)$ and $g(\gamma_d)$, where the possible approximation error is due to the limited accuracy of the interior point method. The incurred error in the computation of $g(\gamma_c)$ and $g(\gamma_d)$ may be aggregated and result in wrong comparisons between their actual values, thereby violating $\gamma^L\in [\gamma_a,\gamma_b]$. To avoid such wrong comparisons, one needs to ensure that the approximation errors $g_{\mathrm{ap}}(\gamma_c)-g(\gamma_c)$ and $g_{\mathrm{ap}}(\gamma_d)-g(\gamma_d)$ are appropriately controlled at every iteration of the algorithm; this will be shown in the next theorem. In particular, we will show how to control the accuracy of the used interior point method for solving the sub-problems $\mathrm{OPT}_j(\gamma_c)$ and $\mathrm{OPT}_j(\gamma_d)$ in order to ensure $\gamma^L\in [\gamma_a,\gamma_b]$ at every iteration of the algorithm. Define the quantity
	\begin{align}\label{min_diff}
	\Delta_g = \min_{\gamma\in[\gamma_0,\gamma^L-\Delta_\gamma]\cup[\gamma^L,1-\Delta_\gamma)}\left|\frac{g(\gamma+\Delta_\gamma)}{1-(\gamma+\Delta_\gamma)}-\frac{g(\gamma)}{1-\gamma}\right|.
	\end{align}
	According to the Proposition~\ref{prop_unimodal}, the function $\frac{g(\gamma)}{1-\gamma}$ is strictly monotone in the intervals $[\gamma_0,\gamma^L]$ and $[\gamma^L,1)$ which implies that $\Delta_g>0$.
	\begin{sloppypar}
		\begin{theorem}\label{thm:runtime}
			Suppose that the input data for~\eqref{opt_f} can be represented with $D$ digits, and that $\eta_2$ satisfies $D\leq C\log(1/\eta_2)$ for a universal constant $C$. Then, Algorithm~\ref{alg:solve} terminates in $O(L^{3.5}k^7\!n\log(n)\!\log(1/\eta_1)\!\log(1/\eta_2))$ time. In particular:
			\begin{itemize}
				\item[1.] If $\gamma_0\leq 1-\underline{\eta_1}/2$ and $\eta_2\leq \min\left\{\frac{2}{1+\sqrt{5}}\Delta_g\underline{\eta_1},\underline{\eta_1}^2\right\}$, then the algorithm returns a feasible solution with $|\bar{\gamma}-\gamma^L|\leq\underline{\eta_1}/2$. Furthermore,
				\begin{align}\label{upperbound}
				\frac{g_{\mathrm{approx}}(\bar\gamma)}{1-\bar{\gamma}}-\frac{g(\gamma^L)}{1-\gamma^L}\leq \left(\frac{g(\gamma_0)}{2(1-\gamma^L)^2\gamma^L}+2\right)\underline{\eta_1}
				\end{align}
				provided that $\underline{\eta_1}\leq 2(1-\gamma^L)^2$.
				\item[2.] If $\gamma_0>1-\underline{\eta_1}/2$, then the algorithm declares infeasibility.
			\end{itemize}
		\end{theorem}
	\end{sloppypar}
\begin{proof}
	See Appendix~\ref{p_thm:runtime}.
\end{proof}

	\subsection{Bootstrapping:}\label{sec:boot}
	Recall that formulating the optimization problem~\eqref{opt_f} relies on the availability of the upper bound $\bar{\epsilon}$ on the actual estimation error $\epsilon = \max\{\|\hat{A}-\trueA\|_2, \|\hat{B}-\trueB\|_2\}$. It is evident from~\eqref{opt_f} that the performance (and even feasibility) of the proposed control design method heavily relies on the conservativeness of $\bar{\epsilon}$: a large value for $\bar{\epsilon}$ results in more restrictive constraints on the system responses. Although in some applications, an upper bound for $\epsilon$ may be readily available based on the domain knowledge, its value may be too conservative for practical purposes. A simple method to alleviate this issue is to resort to a bootstrap approach, where the goal is to \textit{estimate the estimation error}, merely based on the available data samples. In particular, given the estimates $\hat{A}$ and $\hat{B}$, we draw sample trajectories from the empirical distribution induced by $(\hat{A},\hat{B})$ in $N$ rounds. Using these synthetically generated sample trajectories at each round $i$, we re-estimate the system dynamics $\hat{A}^{(i)}$ and $\hat{B}^{(i)}$. Finally, an upper bound on the estimation error is obtained by setting $\bar{\epsilon}$ as $100\times (1-\delta)$ percentile of $\max\{\|\hat{A}^{(i)}-\hat{A}\|_2, \|\hat{B}^{(i)}-\hat{B}\|_2\}, i= 1,\dots, N$, for some parameter $\delta>0$. Roughly speaking, the obtained estimation error is an upper bound on the actual one with probability of at least $1-\delta$. Similar bootstrap methods are widely used for estimating various characteristics of estimators, such as their bias, variance, etc. A more detailed analysis on bootstrap methods can be found in~\cite{efron1994introduction, hall2013bootstrap, shao2012jackknife}. 
	
	Algorithm~\ref{alg:bootstrap} describes the proposed method for obtaining $\bar{\epsilon}$. In this algorithm, the matrix $M$ is defined as~\eqref{true_M}, where $P$ refers to the solution of the Lyapunov equation~\eqref{true_P} after replacing the true system matrices with the estimated ones.
	
	\begin{algorithm}
		\caption{}
		\label{alg:bootstrap}
		\begin{algorithmic}[1]
			\STATE{{\bf input:} Initial state $x_0$, estimates $\hat{A},\hat{B}$, initial controller $K_0$, distribution parameters $\eta_w$, $\eta_v$, ${M}$, confidence parameter $\delta$, and number of rounds $N$}
			\STATE{{\bf output:} upper bound on the estimation error $\bar{\epsilon}$}
			\FOR{$i$ in $\{1,\dots, N\}$}
			\STATE{$x(0)\sim \mathcal{N}(0,M)$}
			\FOR{$\tau$ in $\{0,\dots, T-1\}$}
			\STATE{$u(\tau)\leftarrow K_0x(\tau)+v(\tau)$, where $v(\tau)\sim\mathcal{N}(0,\eta_v^2I)$}
			\STATE{$x(\tau+1)\leftarrow \hat{A}x(\tau)+\hat{B}u(\tau)+w(\tau)$ where $w(\tau)\sim \mathcal{N}(0,\eta_w^2I)$}
			\ENDFOR
			\STATE{Obtain $(\hat{A}^{(i)},\hat{B}^{(i)})$ by solving $\texttt{LASSO}(1,1:T-1)$ with $\left(\{x(\tau)\}_{\tau = 0}^T, \{u(\tau)\}_{\tau = 0}^{T-1}\right)$ as input}
			\STATE{$\bar{\epsilon}^{(i)}\leftarrow \max\{\|\hat{A}^{(i)}-\hat{A}\|, \|\hat{B}^{(i)}-\hat{B}\|\}$}
			\ENDFOR
			\RETURN{$\bar{\epsilon}$ as the $100\times (1-\delta)$ percentile of $\left\{\bar{\epsilon}^{(i)}\right\}_{i=1}^N$.}
		\end{algorithmic}
	\end{algorithm}

\section{Numerical Results:}
\begin{sloppypar}
	To illustrate the effectiveness of the developed control design framework, we focus on a class of graph Laplacian systems with \textit{chain} structures. Let the scalars $x_i(t)$, $u_i(t)$, and $w_i(t)$ denote the state, input, and the disturbance corresponding to the subsystem $i$.
	Consider the following dynamics:
\end{sloppypar}
%
\vspace{-2mm}
\begin{align}\label{ss_chain}
x_i(t+1) &= (D_i\!+\!1\!-\!2a_i) x_{i}(t) \!+\! a_i(x_{i-1}(t)\!+\!x_{i+1}(t)) \!+\! b_i u_i(t) \!+\! w_i(t) && \text{if}\ 2\leq i\leq n-1\nonumber\\
x_i(t+1) &= (D_i+1-a_i) x_{i}(t) + a_i x_{i-1}(t) + b_i u_i(t) + w_i(t) && \text{if}\ i= n\nonumber\\
x_i(t+1) &= (D_i+1-a_i) x_{i}(t) + a_i x_{i+1}(t) + b_i u_i(t) + w_i(t) && \text{if}\ i= 1\nonumber\\
\end{align}
\begin{figure}
	\centering
	\includegraphics[width=.4\columnwidth]{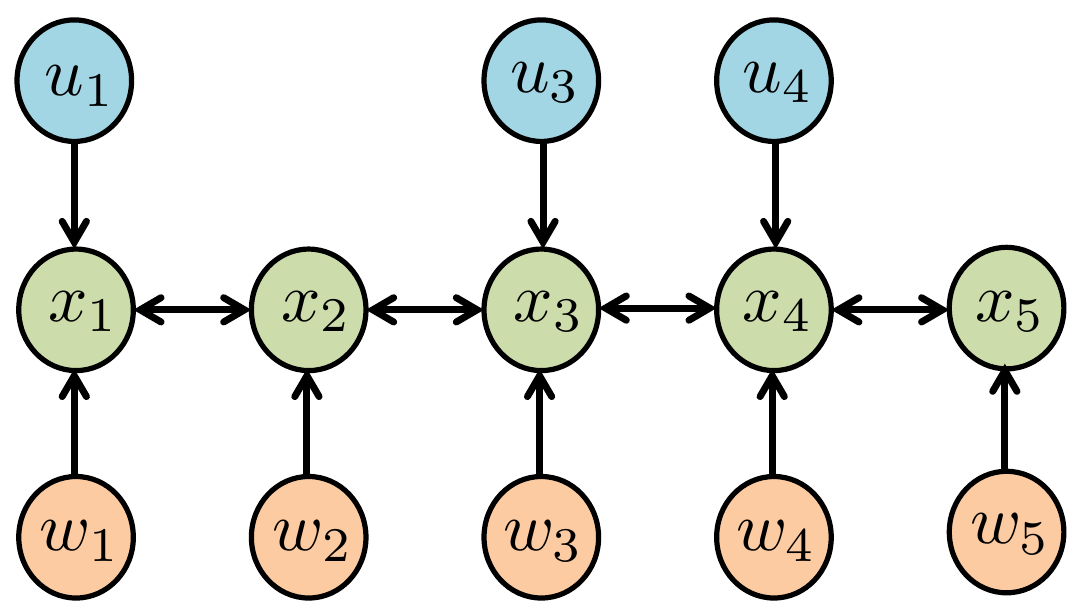}
	\caption{ \footnotesize A realization of the graph Laplacian systems with chain structures. The number of state and input signals are equal to 5 and 3, respectively.}
	\label{chain}
\end{figure}
where $D_i$ and $a_i$ are scalar numbers, and $b_i$ is a binary number taking the value 1 only if subsystem $i$ is directly controlled by an input signal; see Figure~\ref{chain} for a simple realization of this model. We assume that $w(t)\sim\mathcal{N}(0,I)$ in all of our experiments. Inspired by the exponential decay of the truncation error with respect to the FIR length $L$ in Theorem~\ref{thm_FIR}, we set the parameter $\alpha$ in~\eqref{opt_f} to $1.2^{-L}$ throughout our simulations. Similar to~\cite{wang2014localized}, we assume that the control structure is \textit{local} and subject to \textit{communication delays}, both of which can be translated to sparsity constraints on the system responses. In particular, given the locality parameter $d$, we are interested in designing a control structure with the property that the effect of a disturbance signal $w_i(t)$ hitting subsystem $i$ is localized to a region defined by its $d$-hop neighbors. Furthermore, given the communication speed parameter $c$, the sub-controllers can interact $c$ times faster than their corresponding subsystems. In particular, given the subsystems $i$ and $j$ with $b_i = b_j = 1$ and $|i-j| = k$, the control action $u_i(t)$ can use $x_j(\tau)$ and $u_j(\tau)$, provided that $\tau\leq (t-k)/c$. 
The local and communication constraints can be translated into sparsity constraints on the system responses. In particular, define
\begin{align}\label{struct}
&\cC_x(t) = \mathcal{S}\left(\mathrm{supp}(A)^{\min\{d-1,\max\{0,c(t-1)\}\}}\right)\\
&\cC_u(t) = \mathcal{S}\left(\mathrm{supp}(B)^\top\cdot\mathrm{supp}(A)^{\min\{d-1,\max\{0,c(t-1)\}\}}\right)
\end{align}  
for every $t\in\{1,\dots, L\}$. Then, the constraints $\Phi_x(t)\in\cC_x(t)$ and $\Phi_u(t)\in\cC_u(t)$ imply that the resulted controller satisfies the prescribed local and communication constraints. More details on these derivations can be found in~\cite{wang2014localized}. As an example, Figure~\ref{sysres} shows the sparsity patterns of the system responses for $d = 5$ and $c = 2$. 

All the simulations in this section are run on a laptop computer with an Intel Core i7 quad-core 2.50 GHz CPU and 16GB RAM. The reported results are for a serial implementation in MATLAB using the CVX framework and the MOSEK solver with default settings.

\begin{figure*}
	\centering
	\subfloat{\label{sysres1}
		\includegraphics[width=.22\columnwidth]{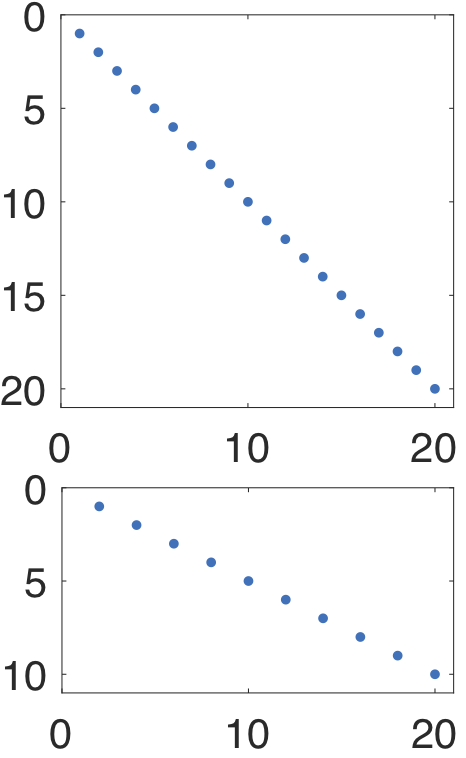}}
	\hspace{0cm}
	\subfloat{\label{sysres2}
		\includegraphics[width=.22\columnwidth]{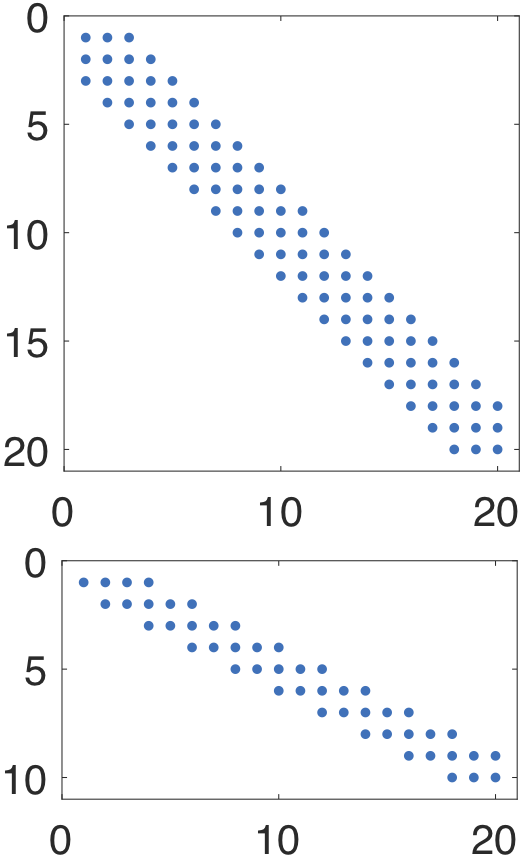}}
	\hspace{0cm}
	\subfloat{\label{sysres3}
		\includegraphics[width=.216\columnwidth]{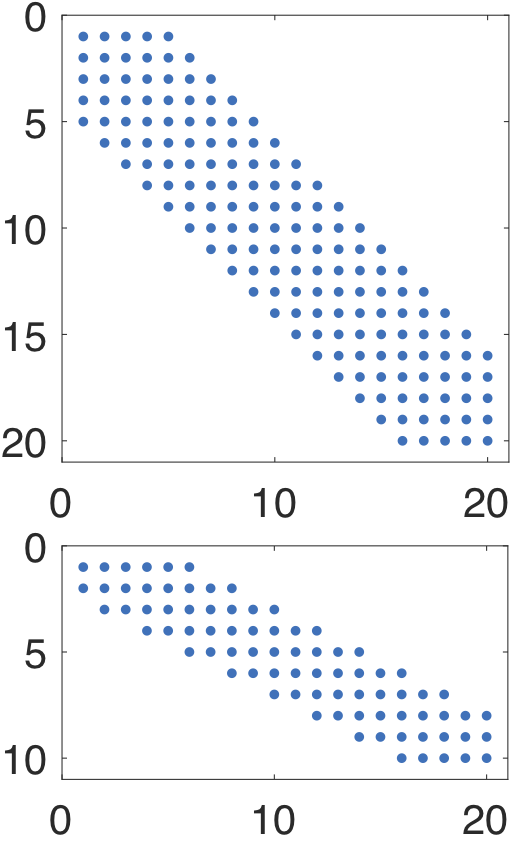}}
	\hspace{0cm}
	\subfloat{\label{sysres4}
		\includegraphics[width=.22\columnwidth]{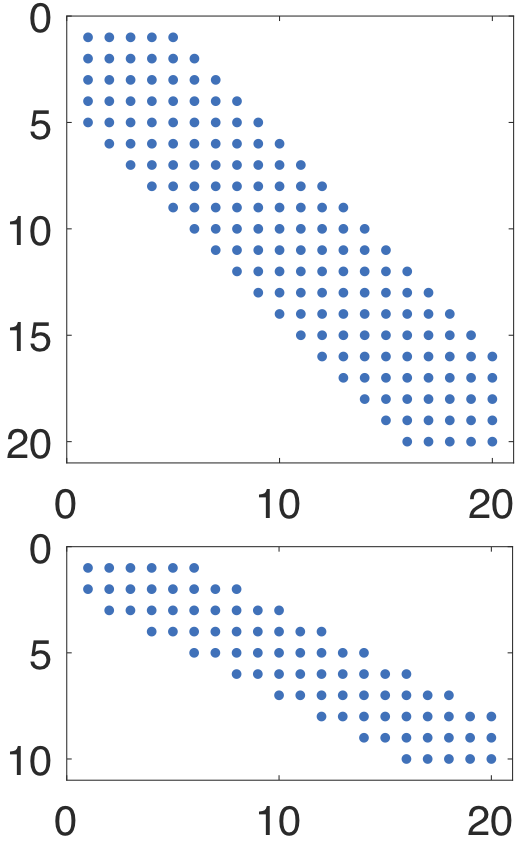}}
	
	\caption{ \footnotesize The sparsity pattern of the system responses $\{\Phi_x(t)\}_{t=1}^4$ and $\{\Phi_u(t)\}_{t=1}^4$ when $d = 5$ and $c = 4$. We assume that $n = 20$ and $b_i = 1$ for every other sub-system. The top row (from left to right) shows the sparsity patterns of $\Phi_x(1),\dots, \Phi_x(4)$. The bottom row (from left to right) shows the sparsity patterns of $\Phi_u(1),\dots, \Phi_u(4)$.}
	\label{sysres}
\end{figure*}

\subsection{Stability analysis} In the first experiment, we consider a small-scale instance of the problem and study the robustness of the designed controller with respect to the uncertainties in the model. In particular, the considered system has $8$ states, $m$ of which are randomly chosen and equipped with input signals, for $m \in\{5,6,7,8\}$. We choose $a_i = 1/3$ for every $i\in\{1,\dots, 8\}$. In order to make the open-loop system marginally unstable, we set $D_i = 0.05$ for $i\in\{2,\dots,7\}$ and $D_1 = D_8 = 0.05-1/3$. We also assume $10\%$ element-wise uncertainty in the estimated system matrices $\hat{A}$ and $\hat{B} $. In other words, $\hat{A}_{ij}$ is randomly chosen from the interval $[{\trueA}_{ij}-0.1|{\trueA}_{ij}|, {\trueA}_{ij}+0.1|{\trueA}_{ij}|]$  for every $(i,j)\in\{1,\dots,8\}^2$. Similarly, $\hat{B}_{kl}$ is randomly chosen from the interval $[{\trueB}_{kl}-0.1|{\trueB}_{kl}|, {\trueB}_{kl}+0.1|{\trueB}_{kl}|]$ for every $(k,l)\in\{1,\dots,8\}\times \{1,\dots, m\}$. Finally, assume that the estimation error $\epsilon=\max\{\|\hat{A}-A_\star\|_2, \|\hat{B}-B_\star\|_2\}$ is known. Later, we will relax these assumptions and estimate $\hat{A}$, $\hat{B}$, and $\epsilon$ directly from the sample trajectories, using the system identification and bootstrap methods that are introduced in Subsections~\ref{sec:sysID} and~\ref{sec:boot}.
The FIR length $L$ is set to $10$. Finally, we set the locality parameter $d$ and the communication speed parameter $c$ to $3$ and $2$, respectively.

The goal in this simulation is to illustrate the robustness of the introduced distributed controller, compared to the \textit{nominal} distributed (designed based on localized SLS approach in~\cite{wang2014localized}) and centralized controllers (designed using Ricatti equations) that treat $\hat{A}$ and $\hat{B}$ as the true parameters of the system without taking into account their estimation errors.\footnote{Note that the nominal controller is also known as \textit{certainty equivalent controller} in the literature; see~\cite{aastrom2013adaptive,mania2019certainty}.} For each input dimension $m\in\{5,6,7,8\}$, we generate 100 independent instances of the problem and design the robust distributed, nominal distributed, and nominal centralized controllers. Figure~\ref{fig_stab} shows the ratio of the instances for which each controller stabilizes the system. As can be seen, the proposed robust distributed controller outperforms the nominal distributed controller when $m$ is equal to 6,7, and 8. In particular, the nominal distributed controller either did not exist or failed to stabilize the true system for $100\%$ and $98\%$ of the instances when $m$ is equal to 6 and 7, significantly underperforming compared to the robust distributed controller. Furthermore, the decrease in $m$ deteriorated the performance of the nominal and robust distributed controllers. In particular, for $m = 5$, both controllers ceased to exist for all of the instances. This is indeed not a surprising observation: roughly speaking, designing a distributed controller with restrictive conditions on its locality and communication speed becomes harder as the input dimension decreases. On the other hand, the centralized controller stabilized the true system for $70\%$ of the instances. Notice that this controller is free of local and communication constraints and hence, its success rate is independent of the input dimension. Overall, the proposed robust distributed controller outperforms the nominal distributed and centralized controllers, provided that the input dimension is not too small. 

Another benefit of the proposed controller compared to its nominal counterparts is its ability to identify whether there is ``too much uncertainty'' in the model. In particular, the infeasibility of the proposed optimization problem~\eqref{opt_f} implies that the estimation error in the model is too large to be accommodated by a robust controller; indeed, such information cannot be inferred by a nominal controller since it is oblivious to the uncertainties in the model.

\begin{figure}
	\centering
	\includegraphics[width=.5\columnwidth]{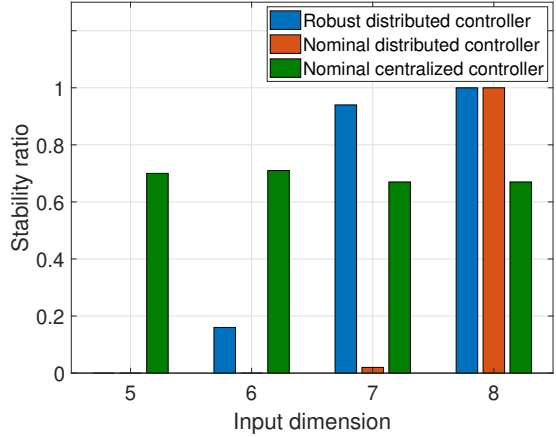}
	\caption{ \footnotesize The ratio of the robust distributed, nominal distributed, and nominal centralized controllers that stabilize the true system.}
	\label{fig_stab}
\end{figure}

\subsection{End-to-end performance}

\begin{sloppypar}
	Next, we showcase the end-to-end performance of the proposed robust distributed controller in larger systems. Given a graph Laplacian system, we assume that its dynamics are unknown and first identify the system matrices with a single sample trajectory using the proposed Lasso-based estimator~\eqref{lasso}. Then, we obtain an upper bound on the estimation error using the bootstrap method introduced in Algorithm~\ref{alg:bootstrap}. Finally, we design the robust distributed controller using Algorithm~\ref{alg:solve}. 
\end{sloppypar}

Consider the system dynamics~\eqref{ss_chain} with $n = 40$, where each subsystem is equipped with an input signal (i.e. $\trueB = I$). Assume that $D_i = 0$ and $a_i = 0.2$ for every $i\in\{1,\dots,n\}$. We further multiply the resulting matrix $\trueA$ by $0.99$ in order to make it marginally stable. To identify the dynamics, we excite the system with a sequence of randomly generated input signals $u(t)\sim \mathcal{N}(0, 0.1I)$ for $t = 0,1,\dots, T$. The initial controller $K_0$ is set to zero since the open-loop system is stable. After estimating the system dynamics, we obtain the bootstrapped estimation error using Algorithm~\ref{alg:bootstrap} with the confidence parameter $\delta = 0.05$ and the number of rounds $N = 500$. 

Figure~\ref{est} shows the true and bootstrapped estimation errors with respect to the learning time $T$. It can be seen that the bootstrapped error is a reliable upper bound on the true estimation error. Given the estimated system matrices and the bootstrapped error, we design the robust distributed controller using Algorithm~\ref{alg:solve}. Figure~\ref{perform} illustrates the end-to-end performance of the designed controller with respect to the learning time $T$ and for different FIR lengths $L$, compared to the oracle cost\footnote{To obtain the oracle cost, we solved the oracle optimization~\eqref{eq22} to near-optimality after restricting the system responses to FIR filters with length 100. We empirically observed that a further increase in the FIR length has little to no effect on the controller cost}. It can be seen that the designed distributed controller performs similarly to the oracle one, even when learning time $T$ is as short as $150$, which is approximately equal to the number of nonzero elements in $(\trueA,\trueB)$. Furthermore, the performance of the controller improves as the estimation error shrinks or, equivalently, the learning time increases. Furthermore, there is a non-negligible improvement in the performance of the designed controller if the FIR length is increased from 4 to 8. However, the improvement in performance is marginal if the FIR length is increased from 8 to 12, indicating that the $L = 8$ is a reasonable choice for the designed distributed controller.

\begin{figure*}
	\centering
	\subfloat[Estimation errors]{\label{est}
		\includegraphics[width=.48\columnwidth]{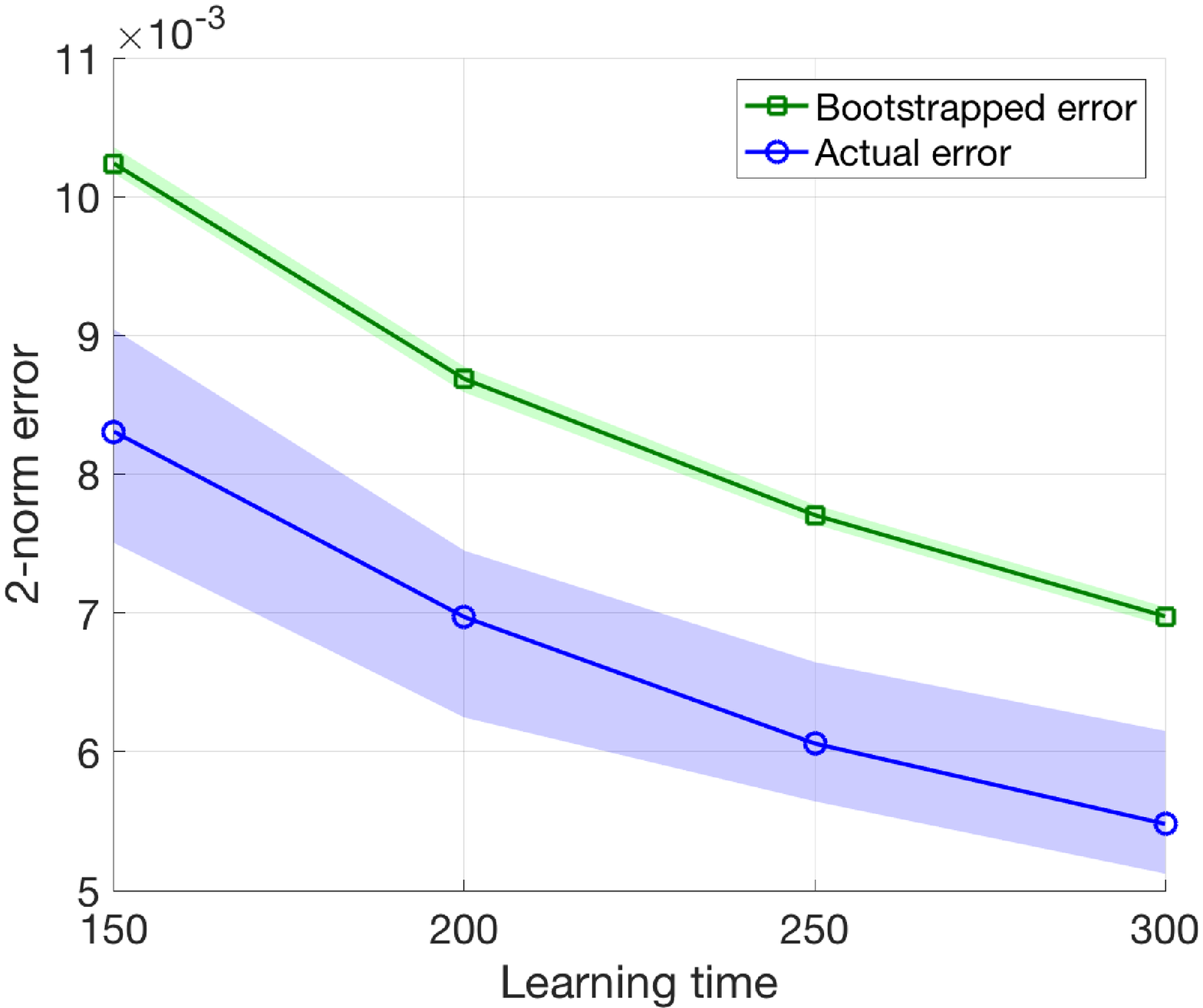}}
	\hspace{0cm}
	\subfloat[Performance]{\label{perform}
		\includegraphics[width=.48\columnwidth]{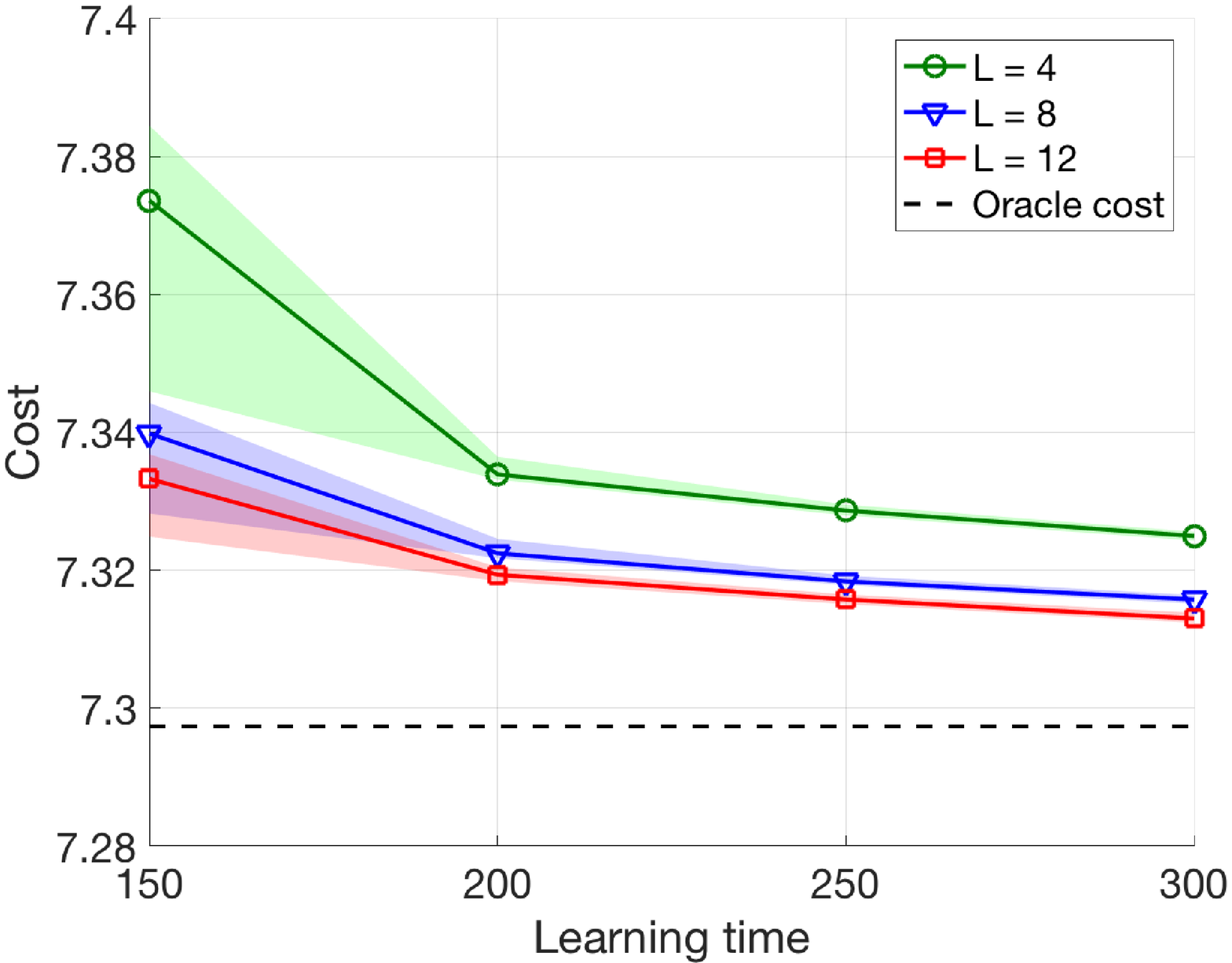}}
	
	\caption{ \footnotesize (a) The true and bootstrapped estimation errors with respect to the learning time. (b) The end-to-end performance of the designed robust distributed controller with respect to learning time and for different FIR lengths. The shaded areas show the quartiles.}
	\label{fig2}
\end{figure*}

Finally, we evaluate the runtime of Algorithm~\ref{alg:solve} for different system dimensions. Consider the same dynamics for the system as before, with $n$ changing from $20$ to $150$. Figure~\ref{fig_runtime} shows the empirical runtime of the proposed algorithm. A $\log$-$\log$ regression yields an empirical time complexity of $\mathcal{O}(n^{1.004})$ for the algorithm, being in line with the theoretical time complexity of the algorithm in Theorem~\ref{thm:runtime}. Finally, it is worthwhile to mention that Algorithm~\ref{alg:solve} is highly parallelizable. In particular, given a machine with $n$ cores, the sub-problems in Algorithm~\ref{alg:solve} can be solved in parallel and, consequently, the complexity of the proposed algorithm becomes \textit{independent} of the system dimension.

\begin{figure}
	\centering
	\includegraphics[width=.5\columnwidth]{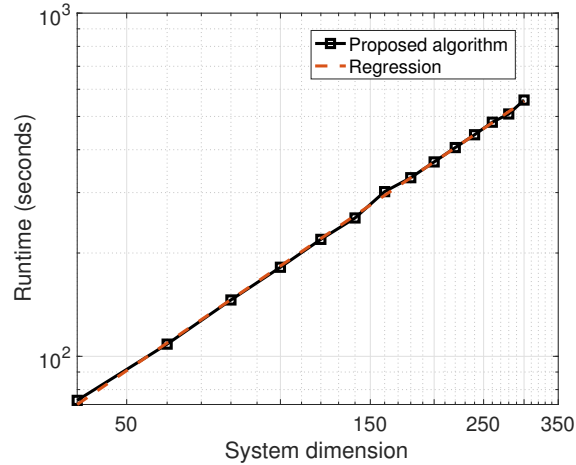}
	\caption{ \footnotesize The empirical runtime of the algorithm with respect to the system dimension (i.e., $n+m$), along with its $\log$-$\log$ regression.}
	\label{fig_runtime}
\end{figure}

\section{Conclusion}
We propose a two-step procedure for designing robust distributed controllers for systems with unknown linear and time-invariant dynamics. Our method first actively probes the system to \textit{learn} a model, and then designs a \textit{robust} distributed controller by taking into account the uncertainty of the learned model.  By taking advantage of recently-developed sparsity-promoting techniques in system identification, together with the localized System Level Synthesis (SLS) framework, we propose the first stabilizing and learning-based distributed controller with guaranteed sub-linear sample complexity and near-linear (constant order if we assume parallel computation) computational complexity. The graceful scalability of the proposed method makes it particularly useful for the control of large-scale and unknown systems with sparse interconnections.

\section*{Acknowledgments}
We are thankful to Javad Lavaei and Han Feng for their insightful comments. This work was supported by the ONR Award N00014-18-1-2526, NSF Award 1808859 and AFSOR Award FA9550-19-1-0055.

	\bibliographystyle{IEEEtran}
	\bibliography{learning-bib,distributed-bib,bib}
	
	\appendix
	
	\section{Proof of Theorem~\ref{thm_FIR}}\label{p_thm_FIR}
	To prove Theorem~\ref{thm_FIR}, we consider the following operator
	\begin{equation}
	\|\tf G\|_{\E1} = \sup_{z \in \mathbb{T}} \: \norm{\tf G(z)}_1 \:
	\end{equation}
	for every $\tf G\in \mathcal{RH}_\infty$.  
	The next lemma describes useful properties of the above operator.
	
	
	\begin{lemma}\label{lem_useful}
		The following statements hold:
		\begin{itemize}
			\item[1.] (Semi-norm property) The operator $\|\cdot\|_{\E1}$ is a well-defined semi-norm on $\mathcal{RH}_{\infty}$.
			\item[2.] (Sub-multiplicativity) For $\tf G, \tf H \in \mathcal{RH}_{\infty}$, we have $\|\tf G\tf H\|_{\E1}\leq \|\tf G\|_{\E1}\|\tf H\|_{\E1}$.
			\item[3.] (H\"older's Inequality) For $\tf G\in \mathcal{RH}_{\infty}$, we have $\|\tf G\|_{\hinf}\leq\sqrt{\|\tf G\|_{\E1}\|\tf G^\top\|_{\E1}}$.
			\item[4.] For $\tf G\in \mathcal{RH}_{\infty}$, we have $\|\tf G\|_{\hinf}\leq\sqrt{k}{\|\tf G\|_{\E1}}$, where $k$ is the maximum number of nonzero elements in different rows of $\tf G$.
			\item[5.] For $\tf G\in \mathcal{RH}_{\infty}$, we have $\|\tf G\|_{\E1}\leq\sum_{t=0}^{\infty}\|G(t)\|_{1}$.
		\end{itemize}
	\end{lemma}
	\begin{proof}
		The first statement follows immediately from the definition of $\|\cdot\|_{\E1}$. Consider the following properties of the induced norms for matrices:
		\begin{itemize}
			\item[$i.$] $\|\tf G(z)\tf H(z)\|_1\leq \|\tf G(z)\|_1\|\tf H(z)\|_1$ for every $z\in\mathbb{T}$.
			\item[$ii.$] $\|\tf G(z)\|_2\leq \sqrt{\|\tf G(z)\|_1\|\tf G(z)^\top\|_1}$ for every $z\in\mathbb{T}$.
			\item[$iii.$] $\|\tf G(z)\|_1\leq k\|\tf G(z)^\top\|_1$ for every $z\in\mathbb{T}$.
		\end{itemize}
		The second, third, and forth statements of the lemma are followed respectively from $(i)$, $(ii)$, and $(iii)$ combined with $(ii)$, respectively. To show the validity of the last statement, note that
		\begin{equation}
		\|\tf G\|_{\E1}\leq\sup_{z \in \mathbb{T}}\left\|\sum_{t=0}^{\infty}G(t)z^{-t}\right\|_1\leq\sup_{z \in \mathbb{T}}\sum_{t=0}^{\infty}\left\|G(t)z^{-t}\right\|_1\leq \sum_{t=0}^{\infty}\left\|G(t)\right\|_1
		\end{equation}
	\end{proof}
	We provide the proof for Theorem~\ref{thm_FIR} in two steps:
	\begin{itemize}
		\item[1.] We derive conditions under which a feasible solution to~\eqref{opt_f} can be constructed based on the optimal solution of the oracle optimization.
		\item[3.] We derive the gap between the cost of the designed feasible solution and the oracle cost in terms of $\bar{\epsilon}$ and $L$. The obtained gap will be used to derive an upper bound on the optimality gap of the synthesized distributed controller. 
	\end{itemize}
	The following Lemma characterizes a feasible solution to~\eqref{opt_f} based on the system responses of the oracle controller.
	\begin{lemma}\label{lem_feas_FIR}
		Suppose that
		\begin{equation}\label{lower}
		\bar\epsilon<\frac{(1-\rho_\star)\min\{\alpha,1-\alpha\}}{16C_\star\rho_\star} k^{-2}, \qquad L> \frac{2\log(k)+\log\left(\frac{2\sqrt{2}(\|\trueA\|_\infty+\|\trueB\|_\infty)}{1-\alpha}\right)}{1-\rho_\star}
		\end{equation}
		and that $(\hat{A}, \hat{B})$ has the same sparsity as $(A, B)$. Then,
		\begin{subequations}
			\begin{align}
			&\tilde{\Phi}_x(t) = \Phi^\star_x(t),\qquad t=1,\dots,L\\
			&\tilde{\Phi}_u(t) = \Phi^\star_u(t),\qquad t=1,\dots,L\\
			&\tilde{V}(t) = \left\{
			\begin{array}{ll}
			0 & \text{if}\quad t=0\\
			-\Delta_A\Phi_x^\star(t)-\Delta_B\Phi_u^\star(t) & \text{if}\quad t=1,\dots,L-1\\
			-\hat{A}\Phi^\star_x(L)-\hat{B}\Phi^\star_u(L) & \text{if}\quad t = L
			\end{array} 
			\right.\\
			&\tilde{\gamma} = \frac{2C_\star\rho_\star}{1-\rho_\star}\left(\frac{1}{\alpha} k^{3/2}+\frac{2\sqrt{2}}{1-\alpha}k^2\right)\bar\epsilon+\frac{\sqrt{2}}{1-\alpha}\cdot (\|\trueA\|_\infty+\|\trueB\|_\infty)C_\star k^2\rho_\star^L,
			\end{align}
		\end{subequations}
		is feasible for~\eqref{opt_f}.
	\end{lemma}
	
	\begin{proof}
		To show the feasibility of the proposed solution, first note that~\eqref{lower} results in
		\begin{align}
		\frac{2C_\star\rho_\star}{1-\rho_\star}\left(\frac{1}{\alpha} k^{3/2}+\frac{2\sqrt{2}}{1-\alpha}k^2\right)\bar\epsilon<1/2,\quad \frac{\sqrt{2}}{1-\alpha}\cdot (\|\trueA\|_\infty+\|\trueB\|_\infty)C_\star k^2\rho_\star^L<1/2
		\end{align}
		where, in the second inequality, we used the relation $-\log(\rho_*)\geq 1-\rho_\star$. This implies that $\tilde{\gamma}<1$. Furthermore, the definition of $(\tilde{\Phi}_x(t), \tilde{\Phi}_u(t),\tilde{V}(t))$ can be used to show that the constraints~\eqref{opt_f_1},~\eqref{opt_f_2},~\eqref{opt_f_3},~\eqref{opt_f_6},~\eqref{opt_f_7} are satisfied. It remains to show the feasibility of~\eqref{opt_f_4} and~\eqref{opt_f_5}. One can write
		\begin{sloppypar}
			\begin{align}\nonumber
			\max_{j}\sum_{t=0}^{L}\|\tilde V_{:,j}(t)\|_1 \leq& (\|\hat{A}\|_\infty\|\Phi_x^\star(L)\|_1\!+\!\|\hat{B}\|_\infty\|\Phi_u^\star(L)\|_1)\!+\!\sum_{t=1}^{L-1}\epsilon\left(\|\Phi_x^\star(t)\|_1+\|\Phi_x^\star(t)\|_1\right)\nonumber\\
			\leq&(\|\trueA\|_\infty+\|\trueB\|_\infty+2\bar\epsilon)kC_\star\rho_\star^L+\frac{2C_\star\rho_\star}{1-\rho_\star}k\bar\epsilon\nonumber\\
			\leq& (\|\trueA\|_\infty+\|\trueB\|_\infty)kC_\star\rho_\star^L+\frac{4C_\star\rho_\star}{1-\rho_\star}k\bar\epsilon\nonumber\\
			\leq&\frac{1-\alpha}{\sqrt{2}}k^{-1}\tilde\gamma\nonumber\\
			\leq& (1-\alpha)k_v^{-1/2}\tilde{\gamma}
			\end{align}
		\end{sloppypar}
		where, in the last inequality, we used the fact that $k_v\leq 2k^2$. Similarly, we have
		\begin{align}
		\sum_{t=1}^{L}\left\|\begin{bmatrix}
		\bar\epsilon\Phi_x(t)\\
		\bar\epsilon\Phi_u(t)
		\end{bmatrix}_{:,j}\right\|_1&\leq\left(\sum_{t=1}^{L}\|\Phi^\star_x(t)\|_1+\|\Phi^\star_u(t)\|_1\right)\bar\epsilon\nonumber\\
		&\leq\frac{2C_\star\rho_\star}{1-\rho_\star}k\bar\epsilon\nonumber\\
		&\leq \alpha k^{-1/2}\tilde{\gamma}\nonumber\\
		&\leq \alpha k_\phi^{-1/2}\tilde{\gamma}
		\end{align}
		where we used the fact that $k_\phi\leq k$. This completes the proof.
	\end{proof}
	Now we are ready to present the proof of Theorem~\ref{thm_FIR}.
	
	\begin{sloppypar}
		\textit{Proof of Theorem~\ref{thm_FIR}:}  Let $\left(\gamma^L, \left\{{\Phi}^L_x(t)\right\}, \left\{{\Phi}^L_u(t)\right\}, \left\{V^L(t)\right\}\right)$ be the optimal solution of~\eqref{opt_f}. Consider the transfer functions $\bPhi^L_x = \sum_{t=1}^{L}\Phi^L_x(t)z^{-t}$, $\bPhi^L_u = \sum_{t=1}^{L}\Phi^L_u(t)z^{-t}$, and $\bV^L = \sum_{t=0}^{L}V^L(t)z^{-t}$.
		Define ${\bDelta}^{L} = \Delta_A{\bPhi}^{L}_x+\Delta_B{\bPhi}^{L}_u+{\bV}^{L}$. One can easily verify that
	\end{sloppypar}
	\begin{align}
	\begin{bmatrix}
	zI-\trueA & -\trueB
	\end{bmatrix}\begin{bmatrix}
	{\bPhi}^{L}_x\\
	{\bPhi}^{L}_u
	\end{bmatrix} = I+{\bDelta}^{L}
	\end{align}
	Now, we show that $\|{\bDelta}^{L}\|_{\Hinf}<1$. To this end, we write
	\begin{align}\label{hatDelta}
	\|{\bDelta}^{L}\|_{\Hinf}&\leq \|\Delta_A{\bPhi}^{L}_x+\Delta_B{\bPhi}^{L}_u\|_{\Hinf}+\|{\bV}^{L}\|_{\Hinf}\nonumber\\
	&\leq \left\|\begin{bmatrix}
	\frac{\Delta_A}{\bar\epsilon} & \frac{\Delta_B}{\bar\epsilon}
	\end{bmatrix}\right\|_{2}\left\|\begin{bmatrix}
	\bar\epsilon{\bPhi}^{L}_x\\
	\bar\epsilon{\bPhi}^{L}_u
	\end{bmatrix}\right\|_{\Hinf}+\|{\bV}^{L}\|_{\Hinf}\nonumber\\
	&\overset{(a)}{\leq}\!\left({\left\|\begin{bmatrix}
		\bar\epsilon{\bPhi}^{L}_x\\
		\bar\epsilon{\bPhi}^{L}_u
		\end{bmatrix}\right\|_{\E1}\!\left\|\begin{bmatrix}
		\bar\epsilon{\bPhi}^{L}_x\\
		\bar\epsilon{\bPhi}^{L}_u
		\end{bmatrix}^\top\right\|_{\E1}}\right)^{1/2}+\!\left({\|{\bV}^{L}\|_{\E1}\!\|{\bV^{L}}^\top\|_{\E1}}\right)^{1/2}\nonumber\\
	&\overset{(b)}{\leq} k_{\phi}^{1/2}\left\|\begin{bmatrix}
	\bar\epsilon{\bPhi}^{L}_x\\
	\bar\epsilon{\bPhi}^{L}_u
	\end{bmatrix}\right\|_{\E1}+k_v^{1/2}\|{\bV^{L}}\|_{\E1}\nonumber\\
	&\overset{(c)}{\leq} k_{\phi}^{1/2}\max_{j}\left\{\sum_{t=1}^{L}\left\|\begin{bmatrix}
	\epsilon\Phi^L_x(t)\\
	\epsilon\Phi^L_x(t)
	\end{bmatrix}_{:,j}\right\|_1\right\}+k_v^{1/2}\max_{j}\left\{\|V^L_{:,j}(t)\|_1\right\}\nonumber\\
	&\leq\alpha{\gamma}^L+(1-\alpha){\gamma}^L\nonumber\\
	&={\gamma}^L<1
	\end{align} 
	where $(a)$, $(b)$, and $(c)$ are due to Lemma~\ref{lem_useful} and the fact that the maximum number of nonzero elements in different rows of $\begin{bmatrix}
	\Phi_x^L(t)^\top & \Phi_u^L(t)^\top
	\end{bmatrix}^\top$ and $V^L(t)$ is upper bounded by $k_{\phi}$ and $k_v$, respectively. Together with Theorem~\ref{thm:robust}, this implies that the derived controller $\bK^L = \bPhi^L_u{\bPhi^L_x}^{-1}$ stabilizes the true system. The rest of the proof is devoted to verifying the optimality gap for the designed controller $\bK^L$. Based on~\eqref{hatDelta} and Lemma~\ref{lem_cost}, one can write
	\begin{align}
	J(\trueA,\trueB,{\bK}^{L}) &= \left\|\begin{bmatrix}
	Q^{1/2} & 0\\
	0 & R^{1/2}
	\end{bmatrix}\begin{bmatrix}
	{\bPhi}^{L}_x\\
	{\bPhi}^{L}_u
	\end{bmatrix}(I+\bDelta^{L})^{-1}\right\|_{\H2}\nonumber\\
	&\leq\frac{1}{1-\|{\bDelta}^{L}\|_{\Hinf}}\left\|\begin{bmatrix}
	Q^{1/2} & 0\\
	0 & R^{1/2}
	\end{bmatrix}\begin{bmatrix}
	{\bPhi}^{L}_x\\
	{\bPhi}^{L}_u
	\end{bmatrix}\right\|_{\H2}\nonumber\\
	&\leq\frac{1}{1-{\gamma}^{L}}\left\|\begin{bmatrix}
	Q^{1/2} & 0\\
	0 & R^{1/2}
	\end{bmatrix}\begin{bmatrix}
	{\bPhi}^{L}_x\\
	{\bPhi}^{L}_u
	\end{bmatrix}\right\|_{\H2}
	\end{align}
	Now, consider the transfer functions $\tilde \bPhi_x = \sum_{t=1}^{L}\tilde \Phi_x(t)z^{-t}$ and $\tilde \bPhi_u = \sum_{t=1}^{L}\tilde \Phi_u(t)z^{-t}$, where $\tilde\Phi_x(t)$ and $\tilde\Phi_u(t)$ are defined in Lemma~\ref{lem_feas_FIR}. 
	One can write
	\begin{align}
	\frac{1}{1-{\gamma}^{L}}\left\|\begin{bmatrix}
	Q^{1/2} & 0\\
	0 & R^{1/2}
	\end{bmatrix}\begin{bmatrix}
	{\bPhi}^{L}_x\\
	{\bPhi}^{L}_u
	\end{bmatrix}\right\|_{\H2}&\leq \frac{1}{1-\tilde{\gamma}}\left\|\begin{bmatrix}
	Q^{1/2} & 0\\
	0 & R^{1/2}
	\end{bmatrix}\begin{bmatrix}
	\tilde{\bPhi}_x\\
	\tilde{\bPhi}_u
	\end{bmatrix}\right\|_{\H2}\nonumber\\
	&\leq \frac{1}{1-\tilde{\gamma}}J_\star
	\end{align}
	The first inequality is due to the feasibility of $(\tilde\gamma, \tilde{\bPhi}_x, \tilde{\bPhi}_u, \tilde{\bV})$. The second equality is due to the fact that $(\tilde{\bPhi}_x, \tilde{\bPhi}_u)$ are the truncations of the system responses when $\trueK$ acts on the true system to their first $L$ time steps. This implies that
	\begin{equation}
	\frac{J(A,B,{\bK}^{L})-J_\star}{J_\star}\leq \frac{1}{1-\tilde{\gamma}}-1
	\end{equation}
	It remains to obtain an upper bound on the right hand side of the above inequality. We have
	\begin{align}\label{final_bound}
	\frac{1}{1-\tilde{\gamma}}\!-\!1&\!\leq\! \frac{1}{1\!-\!\left(\underbrace{\frac{2C_\star\rho_\star}{1-\rho_\star}\left(\frac{1}{\alpha} k^{3/2}+\frac{2\sqrt{2}}{1-\alpha}k^2\right)\bar\epsilon}_{e_1}\!+\!\underbrace{\frac{\sqrt{2}}{1-\alpha} (\|\trueA\|_\infty\!+\!\|\trueB\|_\infty)C_\star k^2\!\rho_\star^L}_{e_2}\right)}\!-\!1\nonumber\\
	&= \frac{e_1+e_2}{1-e_1-e_2}
	\end{align}
	Using~\eqref{cond}, it is easy to verify that
	we have $e_1\leq 1/4$ and $e_2\leq 1/4$. This implies that 
	\begin{align}
	\frac{J(A,B,{\bK}^{L})-J_\star}{J_\star}\leq 2(e_1+e_2)
	\end{align}
	Plugging back the definitions of $e_1$ and $e_2$, together with some simple algebra completes the proof.$\hfill\square$
	
	\section{Proof of Proposition~\ref{prop_unimodal}}\label{p_prop_unimodal}
	We need a number of lemmas in order to prove this proposition.
	
	\begin{lemma}\label{l_pd}
		Given vectors $a$, $b$, and a positive definite matrix $M$, suppose that $a^\top M a = -a^\top M b = b^\top M b$. Then, we have $a = -b$.
	\end{lemma}
	\begin{proof}
		$a^\top M a = -a^\top M b$ and $b^\top M b = -b^\top M a$ imply $a^\top M (a+b) = 0$ and $b^\top M(a+b) = 0$. Combining these equations leads to $(a+b)^\top M (a+b) = 0$. Due to the positive definiteness of $M$, we have $a = -b$. 
	\end{proof}
	
	\begin{lemma}\label{l_qp}
		For every feasible $\gamma$, $g(\gamma)^2$ can be reformulated as the optimal solution of the following QP:
		\begin{subequations}\label{opt_f2}
			\begin{align}
			\min_{x} &\ \frac{1}{2}x^\top M x\\
			\mathrm{s.t.} &\ H_1x \leq h_1 + \gamma \mathbf{1}\\
			&\ H_2x = 0
			\end{align}
		\end{subequations}
		where 
		\begin{itemize}
			\item[-] $x$ is the vectorized concatenation of $\left( \left\{{\Phi}_x(t)\right\}, \left\{{\Phi}_u(t)\right\}\right)$.
			\item[-] $M$ is a positive definite matrix,
			\item[-] $H_1$ and $H_2$ are matrices that only depend on $(\hat{A}, \hat{B}, \alpha, k)$ and $\cC_v$.
			\item[-] $h_1$ is a vector whose nonzero elements have absolute value greater than 1.
			\item[-] $\mathbf{1}$ is a vector whose elements are equal to 1.
		\end{itemize}
	\end{lemma}
	
	
	\begin{proof}
		The proof follows after writing the slack variables $\{V(t)\}_{t = 0}^L$ in terms of $\{\Phi_x(t)\}_{t=1}^L$ and $\{\Phi_u(t)\}_{t=1}^L$ and linearizing $\ell_1$ norm. The details are omitted for brevity.
	\end{proof}

	\textit{Proof of Proposition~\ref{prop_unimodal}.} According to Lemma~\ref{l_qp}, $g(\gamma)^2$ is equivalent to~\eqref{opt_f2} which is a strictly convex QP. Therefore, based on the result of~\cite{berkelaar1997optimal}, the optimal solution of~\eqref{opt_f2} is a continuous function of $\gamma$ when it is feasible. Therefore, $g(\gamma)^2$ (and hence $g(\gamma)$) is continuous over the interval $[\gamma_0,1)$. By contradiction, suppose that $\frac{g(\gamma)}{1-\gamma}$ is not unimodal. Then, the quasiconvexity of $\frac{g(\gamma)}{1-\gamma}$ in the interval $[\gamma_0,1)$ implies that there must exist $\underline{\gamma}$ and $\bar{\gamma}$ such that $\gamma_0\leq\underline{\gamma}<\bar{\gamma}<1$ and $\frac{g(\gamma)}{1-\gamma}$ is constant in the interval $[\underline{\gamma},\bar{\gamma}]$. This implies that $g(\gamma) = c(1-\gamma)$ and $g(\gamma)^2 = c^2(1-\gamma)^2$ for some $c$ and every $\gamma\in [\underline{\gamma},\bar{\gamma}]$. Define the active set $I(\gamma)$ as the set of the row indices of $H_1$ corresponding to the active inequalities, i.e., the set of indices $i$ for which we have $(H_1)_{i,:}x = (h_1)_i+\gamma$. Let $H_1[I(\gamma)]$ be the submatrix of $H_1$ after removing the rows not belonging to $I(\gamma)$. Without loss of generality, we assume that the matrix $H[I(\gamma)] = \begin{bmatrix}
	H_2^\top & H_1[I(\gamma)]^\top
	\end{bmatrix}^\top$ is full row rank; otherwise, one can remove the dependent rows of $H[I(\gamma)]$ to reduce it to a full row rank matrix. Now, due to the continuity of $x(\gamma)$, there must exist $\underline{\underline{\gamma}}$ and $\bar{\bar{\gamma}}$ such that $\underline{\gamma}\leq \underline{\underline{\gamma}}< \overline{\overline{\gamma}}\leq {\bar{\gamma}}$ and $I(\gamma)$ remains the same for every $\gamma\in[\underline{\underline{\gamma}}, \overline{\overline{\gamma}}]$. Let $I(\gamma)$ be denoted as $I^*$ within this interval. Then,~\eqref{opt_f2} is reduced to
	\begin{align}
	\min_{x} &\ \frac{1}{2}x^\top M x\label{opt_f3}\\
	\mathrm{s.t.} &\ H[I^*]x = h_3[I^*] + \gamma h_4[I^*]
	\end{align}
	for every $\gamma\in[\underline{\underline{\gamma}}, \overline{\overline{\gamma}}]$, where $h_3[I^*] = \begin{bmatrix}
	0 & h_1[I^*]^\top
	\end{bmatrix}^\top$ and $h_4[I^*] = \begin{bmatrix}
	0 & \mathbf{1}[I^*]^\top
	\end{bmatrix}^\top$. We consider two cases:
	
	\noindent{\bf case 1:} Suppose that $I^*$ is empty. This implies that $h_4[I^*] = 0$ and therefore, $g(\gamma)$ is constant over the interval $[\underline{\underline{\gamma}}, \bar{\bar{\gamma}}]$ which is a contradiction. 
	
	\noindent{\bf case 2:} Suppose that $I^*$ is non-empty and hence, $h_4[I^*] \not= 0$. Due to the feasibility of the affine constraints, strong duality holds. Therefore, by solving the dual of~\eqref{opt_f3}, one can explicitly write the optimal value of~\eqref{opt_f3} in the form of 
	\begin{align}
	g(\gamma)^2 =& \frac{1}{2} (h_3[I^*]+\gamma h_4[I^*])^\top \left(H[I^*]M^{-1}H[I^*]^\top\right)^{-1}(h_3[I^*]+\gamma h_4[I^*])\nonumber\\
	=&\frac{1}{2}\left(h_4[I^*]^\top\left(H[I^*]M^{-1}H[I^*]^\top\right)^{-1} h_4[I^*]\right)\gamma^2\nonumber\\
	&+ \left(h_3[I^*]^\top\left(H[I^*]M^{-1}H[I^*]^\top\right)^{-1} h_4[I^*]\right)\gamma\nonumber\\
	&+\frac{1}{2}\left(h_3[I^*]^\top\left(H[I^*]M^{-1}H[I^*]^\top\right)^{-1} h_3[I^*]\right)
	\end{align}
	Since we assumed that $g(\gamma)^2 = c^2(1-\gamma)^2$ for every $[\underline{\underline{\gamma}}, \bar{\bar{\gamma}}]$, the following equalities must be satisfied:
	\begin{align}\nonumber
	h_4[I^*]^\top\left(H[I^*]M^{-1}H[I^*]^\top\right)^{-1} h_4[I^*] =& -h_3[I^*]^\top\left(H[I^*]M^{-1}H[I^*]^\top\right)^{-1} h_4[I^*]\nonumber\\
	=&\ h_3[I^*]^\top\left(H[I^*]M^{-1}H[I^*]^\top\right)^{-1} h_3[I^*]
	\end{align}
	\begin{sloppypar}
		\noindent Note that $\left(H[I^*]M^{-1}H[I^*]^\top\right)^{-1}$ is positive definite due to the fact that $H[I^*]$ is full row rank. Therefore, Lemma~\ref{l_pd} implies that $h_4[I^*] = -h_3[I^*]$. On the other hand, $h_4[I^*]$ has an element with value 1 due to the assumption that $I^*$ is non-empty. Furthermore, according to Lemma~\ref{l_qp}, none of the elements of $h_4$ have magnitude equal to 1. This contradicts with $h_4[I^*] = -h_3[I^*]$ and completes the proof.$\hfill\square$
	\end{sloppypar}

\section{Proof of Theorem~\ref{thm:runtime}}\label{p_thm:runtime}

	\begin{sloppypar}
		First, we show that the algorithm terminates in $O(L^{3.5}k^7n\log(n)\log(1/\eta_1)\log(1/\eta_2))$ time. Without loss of generality, suppose that $g(1)<+\infty$. Then, the \texttt{while} loop will take at most $\lceil\log(1/\eta_1)\rceil$ iterations to satisfy $|\gamma_c-\gamma_d|\leq \eta_1$ and terminate. On the other hand, at each iteration, one needs to solve $\mathrm{OPT}_1(\gamma_c),\dots,\mathrm{OPT}_n(\gamma_c)$ and $\mathrm{OPT}_1(\gamma_d),\dots,\mathrm{OPT}_n(\gamma_d)$ by solving $2n$ instances of the reduced-QPs introduced in Lemma~\ref{l_reducedQP}. Classical results on the interior methods show that each QP can be solved to $\frac{\eta_2}{n}$-accuracy in $O(L^{3.5}k^7\log(n)\log(1/\mathrm{\eta_2}))$~\cite{boyd2004convex, nesterov1994interior}. Combining these time complexities, one can verify that the algorithm terminates in $O(L^{3.5}k^7n\log(n)\log(1/\eta_1)\log(1/\eta_2))$. 
	\end{sloppypar}
	
	Next, we prove the statements 1 and 2 of the theorem.
	
	\noindent{\underline{Proof of statement 2:}} Suppose that $\gamma_0>1-\underline{\eta_1}/2$. Then, it is easy to verify that $\gamma_a$ and $\gamma_b$ will obtain the following values at the end of the \texttt{while} loop:
	\begin{align}
	\gamma_a = 1-\underline{\eta_1},\quad \gamma_b  =1
	\end{align}
	Therefore, $1-\underline{\eta_1}/2$ will be assigned to $\bar\gamma$ after the line 18 of the algorithm. This implies that $\gamma_0>\bar\gamma$ and $g(\gamma) = +\infty$ due to the definition of $\gamma_0$.
	
	\noindent{\underline{Proof of statement 1:}} An argument similar to the proof of the first statement can be used to show that $g(\bar{\gamma})<+\infty$ at the termination of the algorithm. Next, we show that  we have $\gamma^L\in[\gamma_a,\gamma_b]$ at the end of the \texttt{while} loop. This trivially holds if the interior point method that is used to solve $\mathrm{OPT}_i(\gamma_c)$ and $\mathrm{OPT}_i(\gamma_d)$ could achieve zero optimality gap, i.e., $g_{\mathrm{ap}}(\gamma) = g(\gamma)$ at every iteration. As mentioned before, this may not be the case since the values of $g(\gamma)$ are available only up to a nonzero approximation error. 
	By contradiction, suppose $\gamma^L\not\in[\gamma_a,\gamma_b]$ at the end of the \texttt{while} loop. Together with the unimodal property of $\frac{g(\gamma)}{1-\gamma}$, this implies that one of the following events happens before the line 11 of the algorithm in at least one iteration of the \texttt{while} loop:
	\begin{itemize}
		\item[-] $g(\gamma_c)$ and $g(\gamma_d)$ are finite, $\gamma^L\in[\gamma_d,\gamma_b]$, $\frac{g(\gamma_c)}{1-\gamma_c}\geq\frac{g(\gamma_d)}{1-\gamma_d}$, and $\frac{g_{\mathrm{ap}}(\gamma_c)}{1-\gamma_c}<\frac{g_{\mathrm{ap}}(\gamma_d)}{1-\gamma_d}$
		\item[-] $g(\gamma_c)$ and $g(\gamma_d)$ are finite, $\gamma^L\in[\gamma_a,\gamma_c]$, $\frac{g(\gamma_c)}{1-\gamma_c}<\frac{g(\gamma_d)}{1-\gamma_d}$, and $\frac{g_{\mathrm{ap}}(\gamma_c)}{1-\gamma_c}\geq\frac{g_{\mathrm{ap}}(\gamma_d)}{1-\gamma_d}$
	\end{itemize}
	Suppose the first event occurs. In particular, assume that $g(\gamma_c)$ and $g(\gamma_d)$ are finite, $\gamma^L\in[\gamma_d,\gamma_b]$, and $\frac{g(\gamma_c)}{1-\gamma_c}\geq\frac{g(\gamma_d)}{1-\gamma_d}$. It is easy to see that $\gamma_d-\gamma_c>\Delta_\gamma$ due to the definition of $\Delta_\gamma$ in~\eqref{const}. On the other hand, notice that $[\gamma_c,\gamma_d]\subseteq [\gamma_0,\gamma^L]$ and hence, $\frac{g(\gamma)}{1-\gamma}$ is decreasing in $[\gamma_0,\gamma^L]$. Therefore, we have $\frac{g(\gamma_c)}{1-\gamma_c}\geq \frac{g(\gamma_d)}{1-\gamma_d}+\Delta_g$ due to the definition of $\Delta_g$ in~\eqref{min_diff}. This leads to the following series of inequalities:
	\begin{align}\label{ineq}
	\frac{g_{\mathrm{ap}}(\gamma_c)}{1-\gamma_c}\geq\frac{g(\gamma_c)}{1-\gamma_c}\geq \frac{g(\gamma_d)}{1-\gamma_d}+\Delta_g\geq \frac{g_{\mathrm{ap}}(\gamma_d)}{1-\gamma_d}+\left(\Delta_g-\frac{\eta_2}{1-\gamma_d}\right)
	\end{align}
	\begin{sloppypar}
		\noindent where the first and last inequalities are due to the fact that $g_{\mathrm{ap}}\geq g(\gamma_c)$ and $g_{\mathrm{ap}}(\gamma_d)\leq g(\gamma_d)+\eta_2$, respectively. Furthermore, it is easy to verify that $\gamma_d\leq\left(1-\frac{2}{1+\sqrt{5}}\right)\underline{\eta_1}$. Combining this inequality with the assumption $\eta_2\leq \frac{2}{1+\sqrt{5}}\Delta_g\underline{\eta_1}$ leads to
		\begin{align}
		\Delta_g-\frac{\eta_2}{1-\gamma_d}\geq \Delta_g - \frac{1+\sqrt{5}}{2}\frac{{\eta_2}}{\underline{\eta_1}}\geq 0
		\end{align}
	\end{sloppypar}
	\noindent Together with~\eqref{ineq}, these inequalities result in $\frac{g_{\mathrm{ap}}(\gamma_c)}{1-\gamma_c}\geq \frac{g_{\mathrm{ap}}(\gamma_d)}{1-\gamma_d}$ which is a contradiction. A similar argument can be made to show that the second event does not occur. Therefore, we have $\gamma^L\in[\gamma_a,\gamma_b]$ at the end of the \texttt{while} loop and therefore, $|\bar{\gamma}-\gamma^L|\leq \underline{\eta_1}/2$. It remains to show that~\eqref{upperbound} is valid, provided that $\underline{\eta_1}\leq (1-\gamma^L)^2$. One can write
	\begin{align}
	\frac{g_{\mathrm{ap}}(\bar\gamma)}{1-\bar{\gamma}}-\frac{g(\gamma^L)}{1-\gamma^L}\leq \underbrace{\frac{g(\bar\gamma)}{1-\bar{\gamma}}-\frac{g(\gamma^L)}{1-\gamma^L}}_{(a)}+\underbrace{\frac{\eta_2}{1-\bar{\gamma}}}_{(b)}
	\end{align} 
	We provide separate upper bounds for $(a)$ and $(b)$. One can verify that the following relation holds for $(b)$:
	\begin{align}\label{ineq_b}
	\frac{\eta_2}{1-\bar{\gamma}}\leq \frac{2\eta_2}{\underline{\eta_1}}\leq 2\underline{\eta_1}
	\end{align}
	where the first and second inequalities are due to $\bar{\gamma}\leq 1-\underline{\eta_1}/2$ and the assumption $\eta_2\leq \underline{\eta_1}^2$. Next, we provide an upper bound for $(a)$. One can write
	
	\begin{align}
	\frac{g(\bar\gamma)}{1-\bar{\gamma}}-\frac{g(\gamma^L)}{1-\gamma^L}&\leq g(\gamma_0)\left|\frac{1}{1-\gamma^L+(\gamma^L-\bar{\gamma})}-\frac{1}{1-\gamma^L}\right|\nonumber\\
	&\leq g(\gamma_0)\frac{|\gamma^L-\bar{\gamma}|}{(1-\gamma^L+(\gamma^L-\bar{\gamma}))(1-\gamma^L)}\nonumber\\
	&\leq g(\gamma_0)\frac{\underline{\eta_1}/2}{(1-\gamma^L-\underline{\eta_1}/2)(1-\gamma^L)}\label{ineq2}
	\end{align}
	where $\underline{\eta_1}\leq 2(1-\gamma^L)^2$ is used in the second inequality to ensure that the denominator is positive. On the other hand, we have
	\begin{align}
	1-\gamma^L-\underline{\eta_1}/2\geq 1-\gamma^L - (1-\gamma^L)^2\geq (1-\gamma^L)\gamma^L
	\end{align}
	Combining this inequality with~\eqref{ineq2} results in
	\begin{align}
	\frac{g(\bar\gamma)}{1-\bar{\gamma}}-\frac{g(\gamma^L)}{1-\gamma^L}\leq \frac{g(\gamma_0)}{2(1-\gamma^L)^2\gamma^L}\underline{\eta_1}
	\end{align}
	This completes the proof.$\hfill\square$
	
\end{document}